\documentclass[draft]{amsart}
\usepackage{cleveref}
\usepackage{autonum}
\usepackage{mathtools}
\usepackage{bm}
\usepackage{braket}

\newif\ifshowkeys

\showkeysfalse

\ifshowkeys
\usepackage{showkeys}

\makeatletter
 \SK@def\cref#1{\SK@\SK@@ref{#1}\SK@cref{#1}}%
\makeatother
\fi

\theoremstyle{plain}
\newtheorem{theorem}{Theorem}
\newtheorem{pretheorem}{Theorem}

\newtheorem{lemma}{Lemma}
\theoremstyle{definition}

\newtheorem{remark}{Remark}

\crefname{section}{Section}{Sections}

\crefname{theorem}{Theorem}{Theorems}
\crefname{pretheorem}{Theorem}{Theorems}
\crefname{corollary}{Corollary}{Corollaries}
\crefname{lemma}{Lemma}{Lemmas}
\crefname{definition}{Definition}{Definitions}
\crefname{remark}{Remark}{Remarks}
\crefname{example}{Example}{Examples}

\crefformat{equation}{\upshape(#2#1#3)}

\renewcommand{\mod}[1]{(\mathrm{mod}\ #1)}

\renewcommand{\epsilon}{\varepsilon}
\renewcommand{\phi}{\varphi}

\DeclareMathOperator{\Li}{\mathrm{Li}}

\def\showsize#1{\setbox0=\hbox{#1}width=\the\wd0, height=\the\ht0, depth=\the\dp0}

\begin{document}

\title[On the distribution of products of two primes]
{On the distribution of products of two primes}
\author[S. Saad Eddin]{Sumaia Saad Eddin}
\address[S. Saad Eddin]{%
Institute of Financial Mathematics and Applied Number Theory\\
Johannes Kepler University\\
Altenbergerstrasse 69\\4040 Linz\\Austria.}
\email{sumaia.saad\_eddin@jku.at}
\author[Y. Suzuki]{Yuta Suzuki}
\address[Y. Suzuki]{%
Graduate School of Mathematics\\
Nagoya University\\
Chikusa-ku\\Nagoya 464-8602\\Japan.}
\email{m14021y@math.nagoya-u.ac.jp}
\date{}


\maketitle

\begin{abstract}
For a real parameter $r$,
the RSA integers are integers which can be written
as the product of two primes $pq$ with $p<q\le rp$,
which are named after the importance
of products of two primes in the RSA-cryptography.
Several authors obtained the asymptotic formulas
of the number of the RSA integers.
However, the previous results on the number of the RSA integers
were valid only in a rather restricted range
of the parameter $r$.
Dummit, Granville and Kisilevsky found some bias
in the distribution of products of two primes with congruence conditions.
Moree and the first author studied some similar bias in the RSA integers,
but they proved that at least for fixed $r$, there is no such bias.
In this paper, we provide an asymptotic formula for the number of the RSA integers
available in wider ranges of $r$, and give some observations
of the bias of the RSA integers,
by interpolating the results
of Dummit, Granville and Kisilevsky
and of Moree and the first author.
\end{abstract}

\section{Introduction}
\label{section:intro}
Let $\pi_2(x)$ be the number of products of two distinct primes, i.e.\ 
\[
\pi_2(x)\coloneqq\sum_{\substack{pq\le x\\p<q}}1.
\]
For this function $\pi_2(x)$,
Landau \cite{Landau,Landau2} proved
\[
\pi_2(x)\sim\frac{x\log\log x}{\log x}
\]
as $x\to\infty$.
He also proved more precise asymptotic formulas, e.g.
\begin{equation}
\label{Landau_asymp}
\pi_2(x)=\frac{x\log\log x}{\log x}+O\left(\frac{x}{\log x}\right).
\end{equation}
In the RSA cryptography,
products of two distinct primes play an important role.
For a real parameter $r>1$,
Decker and Moree \cite{Decker_Moree} introduced
the RSA integers to be integers which can be written
as the product of two primes $pq$ with $p<q\le rp$,
and they studied the distribution of the RSA integers.
Define 
\[
\pi_{2}(x;r)\coloneqq\# \left\{{pq \le x}:{p<q} \le rp\right\}.
\]
They proved that 
\begin{equation}
\label{Decker_Moree}
\pi_{2}(x;r)
=
\frac{2x\log r}{(\log x)^2}+O\left(\frac{rx\log2r}{(\log x)^3}\right).
\end{equation}
Justus~\cite{Justus} also studied asymptotic behavior
of the number of the RSA integers and of its variants.
In particular, Justus dealt with the case $r$ is not so close to $x$ but also relatively large.
By following Justus' argument~\cite[Theorem~2.1]{Justus} with keeping uniformity, we can prove
\begin{equation}
\label{Justus}
\pi_2(x;r)
=
\frac{x}{\log x}
\left(\log\log xr-\log\log\frac{x}{r}\right)
+O\left(\frac{x}{(\log x)(\log\frac{x}{r})}\right).
\end{equation}
In this result, the main term majorizes the error term
in the range $x^{\delta}\le r\le x/4$ for a fixed $\delta>0$. By \cref{lem:to_Landau} below,
the formula \cref{Justus} is reduced to Landau's result~\cref{Landau_asymp} when $r=x/4$.
However, the formula \cref{Justus} falls short of interpolating
\cref{Landau_asymp} and \cref{Decker_Moree}.

Recently, Moree and the first author \cite[Corollary~1.3]{Moree_SaadEddin}
obtained
\begin{equation}
\label{Moree_SaadEddin}
\pi_{2}(x;r)=F_r(x)+O\left(rx\exp\left(-c\sqrt{\log x}\right)\right),
\end{equation}
where $c>0$ is some absolute constant and
\[
F_r(x)\coloneqq\int_{2r}^{x}
\left(\log\log ur-\log\log\frac{u}{r}\right)
\frac{du}{\log u}.
\]
This gives an approximation better than \cref{Decker_Moree} in the error term aspect.
In their paper,
the authors tried to observe some bias
in the distribution of the RSA integers with some congruence conditions,
motivated by a recently observed bias \cite{Dummit_Granville_Kisilevsky}
in products of two prime numbers with congruence conditions.
Consider the ratio 
\begin{equation}
r(x)= \#\left\{{pq\le x}:{p\equiv q\equiv 3\,\mod{4}}\right\}
/\tfrac{1}{4}\# \left\{ pq\leq x \right\}.
\end{equation}
One would expect that $r(x)$ converges rapidly to $1$.
However, Dummit, Granville and Kisilevsky~ \cite{Dummit_Granville_Kisilevsky}
pointed out by numerical calculations
that this convergence is surprisingly slow.
They proved that such ratios indeed converge rather slowly because of the existence of the secondary main term
(we stated the theorem in a slightly different way,
but it is easy to prove this theorem via the original argument): 
\begin{pretheorem}[{\cite[Theorem~1.1]{Dummit_Granville_Kisilevsky}}]
\label{thm:DGK}
Let $\chi$ be a quadratic character $\mod{Q}$ and $\eta=\pm1$.
Then, for $x\ge4$, we have
\[
\frac
{\#\{{pq\le x}:{\chi(p)=\chi(q)=\eta}\}}
{\#\{{pq\le x}:{(pq,Q)=1}\}}
=
\frac{1}{4}\left(1+H_{\chi,\eta}(x)\right),
\]
where $H_{\chi,\eta}(x)$ is given by
\[
H_{\chi,\eta}(x)
=
\frac{\eta}{\log\log x}
\left(
\mathcal{L}_{\chi}
\left(1+O\left(\frac{1}{\log\log x}\right)\right)
+
O\left(\frac{1}{\log x}\right)\right),\quad
\mathcal{L}_{\chi}
\coloneqq
\sum_{p}\frac{\chi(p)}{p}
\]
and the implicit constant depends on the modulus $Q$.
\end{pretheorem}

In \cite{Moree_SaadEddin},
no such bias was found in the distributions of the RSA integers at least for fixed $r$.
However, since $\pi_2(x;r)$ is reduced to $\pi_2(x)$ for large $r$
as we shall see in \cref{lem:to_Landau} below and there is a bias for the products of two primes,
some bias should show up even for the RSA integers
when $r$ is large enough.
In this sense, the behavior of $\pi_2(x;r)$ in the $r$-aspect is actually important, which is thought to be not very important in \cite{Moree_SaadEddin}.

The preceding studies of the RSA integers
are still not so satisfactory in this $r$-aspect.
First, Justus' formula \cref{Justus} is available in a rather wide range of $r$ but it fails to interpolate  \cref{Landau_asymp} and
\cref{Decker_Moree}.
Namely, Justus' formula \cref{Justus} is not available for $r$ not so large.
On the other hand, as it is already mentioned in \cite{Decker_Moree} and \cite{Moree_SaadEddin},
the asymptotic formulas \cref{Decker_Moree}
and \cref{Moree_SaadEddin} are not available for large $r$ compared to $x$.
For $r\ge2$, the main term of \cref{Moree_SaadEddin} is bounded as
\[
F_r(x)
\ll
\int_{2r}^{x}\frac{\log\log ur}{\log u}du
\ll
x\log\log r.
\]
Thus, in order to make the main term being of larger magnitude than the error term,
we need to assume at least that
\[
x\log\log r\gg rx\exp\left(-c\sqrt{\log x}\right),
\]
which is roughly equivalent to
\[
r\ll\exp\left(c\sqrt{\log x}\right).
\]
By \cref{lem:to_Landau} below,
the RSA integers are reduced to the products of two primes
for $r\ge x/4$, so this restriction can be too strong to
find bias in the RSA integers.
Furthermore, though it is of rather less importance,
the previous asymptotic formulas are not available for $r$ very close to $1$.
Indeed, if $1<r\le2$, then we have
\[
\log r\asymp(r-1)
\]
and if further $x$ is sufficiently large, then
\begin{align}
F_r(x)
&=\int_{2r}^{x}
\left(\log\log ur-\log\log\frac{u}{r}\right)
\frac{du}{\log u}\\
&=2\int_{r^2}^{x}\left(\sum_{\ell=0}^{\infty}\frac{1}{2\ell+1}\left(\frac{\log r}{\log u}\right)^{2\ell+1}\right)\frac{du}{\log u}+O(1)\\
&\asymp(\log r)\int_{2r}^{x}\frac{du}{(\log u)^2}
\asymp\frac{x\log r}{(\log x)^2}\asymp\frac{(r-1)x}{(\log x)^2}.
\end{align}
Thus both of the asymptotic formulas \cref{Decker_Moree} and \cref{Moree_SaadEddin} has the main term of the size
\[
\asymp\frac{(r-1)x}{(\log x)^2}
\]
and the error term estimate of the size
\[
\gg x\exp(-c\sqrt{\log x}).
\]
Therefore, in order to make the main term being of larger magnitude than the error term,
we need to assume at least%
\footnote{It seems that the assumption $(r-1)^{-1}=o(\log x)$
is missing in Corollary 1 of \cite{Decker_Moree}.}
\[
\frac{(r-1)x}{(\log x)^2}\gg x\exp(-c\sqrt{\log x})
\]
which is roughly equivalent to
\[
r\ge1+C\exp(-c\sqrt{\log x})
\]
with some absolute constant $C>0$.

The first main aim of this paper is to obtain asymptotic formulas
for the number of the RSA integers which is valid in wider ranges of $r$.
In order to determine when the bias of the RSA integers appears,
it is desired to interpolate the asymptotic formulas for $\pi_{2}(x;r)$
and Landau's asymptotic formula \cref{Landau_asymp}.
For the case $r$ is large, we have the following
refinement of the result of Moree and the first author:
\begin{theorem}
\label{thm:large_int}
For $1\le r\le x/4$, we have
\[
\pi_{2}(x;r)
=
\int_{4r}^{x}\left(\log\log ur-\log\log\frac{u}{r}\right)\frac{du}{\log u}
+\frac{4r\log\log4r}{\log4r}
+O\left(\frac{x}{\log x}e^{-c\sqrt{\log\frac{x}{r}}}\right),
\]
where the constant $c>0$ and the implicit constant
are absolute.
\end{theorem}

Note that by the same argument as above,
we can find that \cref{thm:large_int} has meaning
only in the range $1+C\exp(-c\sqrt{\log x})\le r\le x/4$
for some absolute constant $C>0$.
We may simplify this theorem to obtain
the following asymptotic formula:

\begin{theorem}
\label{thm:large}
For $1+\exp(-c\sqrt{\log x})\le r\le x/4$, we have
\[
\pi_{2}(x;r)=\frac{x}{\log x}\left(\log\log xr-\log\log \frac{x}{r}\right)
+O\left(\frac{x\log r}{(\log x)^2(\log \frac{x}{r})}\right),
\]
where the constant $c>0$ and the implicit constant
are absolute.
\end{theorem}

For $r$ very close to $1$,
we can still have the following asymptotic formula:

\begin{theorem}
\label{thm:small}
For $1+x^{-5/12}<r\le 3/2$ we have 
\begin{equation}
\pi_{2}(x;r)= \frac{2x\log r}{(\log x)^2}+O\left( \frac{x\log r}{(\log x)^{5/2}}(\log\log x)\right),
\end{equation}
where the implicit constant is absolute.
\end{theorem}

If we let $r=x/4$ and apply \cref{lem:to_Landau} below,
then \cref{thm:large_int} and \cref{thm:large}
are reduced to Landau's asymptotic formula \cref{Landau_asymp}.
Thus, those results give the desired interpolation.
Also, by estimating the error term of \cref{thm:large} as
\[
\frac{x\log r}{(\log x)^2(\log \frac{x}{r})}
\ll
\frac{x}{(\log x)(\log \frac{x}{r})},
\]
we obtain Justus' formula \cref{Justus}.
We remark that
\[
1+\exp(-c\sqrt{\log x})\quad\text{and}\quad
O\left(\frac{x}{\log x}e^{-c\sqrt{\log\frac{x}{r}}}\right)
\]
in \cref{thm:large_int} and \cref{thm:large}
can be improved by using the prime number theorem
with the Vinogradov--Korobov type error term \cite[(12.27)]{Ivic}.
Also, the exponent $5/2$ in the error term of \cref{thm:small}
can be improved to $3$ in the narrower range $r\geq 1+x^{-5/12+\epsilon}$
by using \cite[Lemma~5]{Saffari_Vaughan}
instead of the result of Zaccagnini \cite{Zaccagnini} given below
as \cref{thm:Zaccagnini}.

We can combine \cref{thm:large,thm:small}
to obtain the following uniform result. 
\begin{theorem}
\label{thm:uniform}
We have
\[
\pi_{2}(x;r)=\frac{x}{\log x}\left(\log\log xr-\log\log \frac{x}{r}\right)\left(1+O\left(\frac{1}{\log\log x}\right)\right)
\]
for $1+x^{-5/12}< r\le x/4$, where the implicit constant is absolute.
\end{theorem}

The second main aim of this paper is
to detect some bias of the distribution of the RSA integers for large $r$
and to determine when the bias show up.
Compared to our asymptotic formulas above,
our result on the bias is still rather incomplete.
In particular, we will not study the behavior
of the bias coefficient $\mathcal{L}_{\chi}(s)$ given below.
The bias can appear only for the case $r$ is close to $x$.
Thus, we introduce a change of variable $s\coloneqq x/r$.
Also, since the behavior of the resulting bias
is sensitive to the error term of the prime number theorem,
we assume the prime number theorem of the following form:
for a given positive integer $Q$,
there exists a positive number $K$
and a non-negative locally integrable function $\delta(x)$
defined on $[2,+\infty)$ such that
for any non-principal Dirichlet character $\chi\ \mod{Q}$,
we have
\begin{equation}
\tag{$\mathbf{P}$}
\label{PNT_assump}
\begin{aligned}
\pi(x)
&\coloneqq
\sum_{p\le x}1
=
\Li(x)+O\left(K\frac{x}{\log x}\delta(x)\right),\\
\pi(x,\chi)
&\coloneqq
\sum_{p\le x}\chi(p)
\le
K\frac{x}{\log x}\delta(x)
\end{aligned}
\end{equation}
for $x\ge2$ with
\[
\Li(x)\coloneqq\int_{2}^{x}\frac{du}{\log u}
\]
and that $\delta(x)$ satisfies the conditions
\begin{equation}
\tag{$\bm{\Delta1}$}
\label{delta1}
\frac{1}{K}\le
\frac{x}{\log x}\delta(x)
\le K\frac{y}{\log y}\delta(y)\quad\text{for}\quad
2\le x\le y,
\end{equation}
\begin{equation}
\tag{$\bm{\Delta2}$}
\label{delta2}
\delta(y)(\log y)\le K\delta(x)(\log x)\quad\text{for}\quad
2\le x\le y,
\end{equation}
\begin{equation}
\tag{$\bm{\Delta3}$}
\label{delta3}
\int_{2}^{\infty}\frac{\delta(u)}{u}du
<+\infty.
\end{equation}
Note that for $2\le x\le y$, \cref{delta2} implies
\[
\delta(y)=\frac{\delta(y)\log y}{\log y}
\le
\frac{K\delta(x)\log x}{\log y}
\le
K\delta(x)
\]
so we have
\begin{equation}
\tag{$\bm{\Delta2}'$}
\label{delta21}
\delta(y)\le K\delta(x)\quad\text{for}\quad
2\le x\le y.
\end{equation}
For example, the well-known
de la Vall\'ee Poussin type
and Korobov--Vinogradov type prime number theorems
give
\[
\delta(x)=\exp(-c\sqrt{\log x})\quad\text{and}\quad
\delta(x)=\exp\left(-c\frac{(\log x)^{3/5}}{(\log\log (x+4))^{1/5}}\right),
\]
respectively. Note that the constant $K$ above may depend on the modulus $Q$. Also,
by assuming the generalized Riemann hypothesis (GRH),
we may have
\begin{equation}
\label{delta_GRH}
\delta(x)=x^{-1/2}(\log x)^2.
\end{equation}
It is easy to see that
these three error term estimates satisfy \cref{delta1}, \cref{delta2} and \cref{delta3}. 

Our result on the bias of the RSA integers is the following.

\begin{theorem}
\label{thm:bias}
Let $\chi$ be a quadratic character $\mod{Q}$ and $\eta=\pm1$.
Assume that \cref{PNT_assump} holds with the conditions
\cref{delta1}, \cref{delta2} and \cref{delta3}
on $\delta(x)$.
Then, for $2\le r\le x/4$,
\begin{equation}
\label{bias:asymp}
\frac
{\#\{{pq\le x}:{p<q\le rp,\ \chi(p)=\chi(q)=\eta}\}}
{\#\{{pq\le x}:{p<q\le rp,\ (pq,Q)=1}\}}
=
\frac{1}{4}\left(1+H_{\chi,\eta}(x;r)+O(\delta(\sqrt{x}))\right),
\end{equation}
where by writing $s\coloneqq x/r$,
$H_{\chi,\eta}(x;r)$ is given by
\begin{align}
H_{\chi,\eta}(x;r)
&=
\frac{\eta}{\log\log xr-\log\log\frac{x}{r}}
\left(\mathcal{L}_{\chi}(s)\left(1+O\left(\frac{1}{\log\log x}\right)\right)
+
O\left(\frac{\Delta(\sqrt{s})}{\log x}\right)\right),\\
\mathcal{L}_{\chi}(s)
&\coloneqq
\frac{1}{s}\sum_{p<\sqrt{s}}\chi(p)p
+
\sum_{p\ge\sqrt{s}}\frac{\chi(p)}{p},\quad
\Delta(x)\coloneqq
\delta(x)+\int_{x}^{\infty}\frac{\delta(u)}{u}du
\end{align}
and the implicit constant depends only on $Q$ and $K$.
\end{theorem}

Note that when $r=x/4$,
\cref{thm:bias} is reduced to \cref{thm:DGK}
since in this case,
\[
\log\log xr-\log\log\frac{x}{r}
=
\log\log\frac{x^2}{4}+O(1)
=
\log\log x+O(1),
\]
\[
\frac{\Delta(\sqrt{s})}{\log x}
\ll
\frac{1}{\log x},\quad
\Delta(\sqrt{x})\log\log x
\ll
e^{-c\sqrt{\log x}},\quad
\mathcal{L}_{\chi}(s)=\mathcal{L}_{\chi}
\]
and the condition $q\le rp$ is vacuous in \cref{thm:bias},
where we used the de la Vall\'ee Pouusin type
prime number theorem and $c>0$ is some absolute constant.
In this sense, \cref{thm:bias}
gives a generalization of \cref{thm:DGK}.

An upper bound for ``the bias coefficient''
$\mathcal{L}_{\chi}(s)$ can be obtained through \cref{PNT_assump},
\cref{delta1} and \cref{delta2} as follows.
By partial summation, we may obtain
\[
\frac{1}{s}\sum_{p\le\sqrt{s}}\chi(p)p
=
\frac{1}{s}\int_{2-}^{\sqrt{s}}ud\pi(u,\chi)
=
\frac{\pi(\sqrt{s},\chi)}{\sqrt{s}}
-
\frac{1}{s}\int_{2}^{\sqrt{s}}\pi(u,\chi)du
\ll
\frac{\Delta(\sqrt{s})}{\log s}
\]
and
\[
\sum_{p>\sqrt{s}}\frac{\chi(p)}{p}
=
\int_{\sqrt{s}}^{\infty}\frac{d\pi(u,\chi)}{u}
=
-\frac{\pi(\sqrt{s},\chi)}{\sqrt{s}}
+
\int_{\sqrt{s}}^{\infty}\frac{\pi(u,\chi)}{u^2}du
\ll
\frac{\Delta(\sqrt{s})}{\log s}.
\]
Therefore,
\begin{equation}
\label{bias:coefficient_bound}
\mathcal{L}_{\chi}(s)
=
\frac{1}{s}\sum_{p<\sqrt{s}}\chi(p)p
+
\sum_{p\ge\sqrt{s}}\frac{\chi(p)}{p}
=
\frac{1}{s}\sum_{p\le\sqrt{s}}\chi(p)p
+
\sum_{p>\sqrt{s}}\frac{\chi(p)}{p}
\ll
\frac{\Delta(\sqrt{s})}{\log s}.
\end{equation}
Therefore, for a fixed $r$, i.e.\ in the range $r\ll1$,
\[
H_{\chi,\eta}(x;r)
\ll
\log x
\left(\frac{\Delta(\sqrt{s})}{\log s}
+
\frac{\Delta(\sqrt{s})}{\log x}\right)
\ll
\Delta(\sqrt{s})
\]
and $s\asymp x$. Thus, by using the de la Vall\'ee Poussin type
prime number theorem,
\[
H_{\chi,\eta}(x;r)
\ll
\exp(-c\sqrt{\log x})
\]
for some positive absolute constant $c>0$.
Thus, when $r\ll1$, \cref{thm:bias} is reduced to
a special case of
the conclusion of Moree and the first author
\cite[Corollary~1.4]{Moree_SaadEddin}.

By the above two observations,
we may say that \cref{thm:DGK} and Corollary~1.4 of \cite{Moree_SaadEddin} are interpolated through \cref{thm:bias}.
To conclude the introduction,
we give a discussion on when the bias of the RSA integers appears.
If the upper bound \cref{bias:coefficient_bound} is tight
with the GRH-error term estimate \cref{delta_GRH}, i.e.\ 
if we can prove a similar $\Omega$-results on $\mathcal{L}_{\chi}(s)$,
then in the asymptotic formula of $H_{\chi,\eta}(x;r)$
in \cref{thm:bias}, the error term
\[
O\left(\frac{\Delta(\sqrt{s})}{\log x}\right)
\]
does not supersede the main term $\mathcal{L}_{\chi}(s)$
provied only $\epsilon\log x\ge\log s$ for sufficiently small $\epsilon>0$. Thus, the only term which may affect the main term
is the error term $O(\delta(\sqrt{x}))$
in \cref{bias:asymp}
or, if we combine with the $\log\log xr-\log\log x/r$ factor,
is the error term $O(\delta(\sqrt{x})\log\log x)$.
However, it is easy to see that
$O(\delta(\sqrt{x})\log\log x)$ also has
a magnitude smaller than the main term
just provided $\epsilon\log x\ge\log s$.
Although we still should prove some $\Omega$-results of $\mathcal{L}_{\chi}(s)$
and also we assumed GRH above,
these observations indicate
that some bias of the RSA integers may occur
for $r$ of the size $r\ge x^{1-\epsilon}$
with some suitable small number $\epsilon>0$.
Note that for $4\le s\le s_0$ with a fixed $s_0$,
numerical calculations of $\mathcal{L}_{\chi}(s)$
can be used as a substitution for the $\Omega$-results of $\mathcal{L}_{\chi}(s)$.

\section{Preliminary lemmas}
\label{section:prelim}
Throughout the paper,
we use the following notation and convention.
The letters $p,q$ are reserved for expressing prime numbers.
The letter $c$ denotes positive constants
which may have different values at different occurrences.
If Theorem or Lemma is stated
with the phrase ``where the implicit constant depends only on $a,b,c,\ldots$'',
then every implicit constant in the corresponding proof
may also depend on $a,b,c,\ldots$ even without special mention.

In this section, we prove some preliminary lemmas.

We begin with checking that $\pi_{2}(x;r)$
is reduced to $\pi_2(x)$ for large $r$.

\begin{lemma}
\label{lem:to_Landau}
For $r\ge x/4$, we have $\pi_{2}(x;r)=\pi_2(x)$.
\end{lemma}
\begin{proof}
This is obvious since if $r\ge x/4$, then we have
\[
rp\ge2r\ge\frac{x}{2}\ge\frac{x}{p}\ge q
\]
for any prime numbers $p,q$ with $pq\le x$.
Thus, the condition
\[
pq\le x\quad\text{and}\quad
p<q\le rp
\]
is reduced to
\[
pq\le x\quad\text{and}\quad
p<q,
\]
which is the condition for the counting function $\pi_2(x)$.
\end{proof}

The next lemma is trivial and stated only for reference.

\begin{lemma}
\label{lem:monotone}
Let $c>0$. For $y\ge x\ge1$, we have
\[
x\exp(-c\sqrt{\log x})
\ll
y\exp(-c\sqrt{\log y}),
\]
where the implicit constant depends only on $c$.
\end{lemma}
\begin{proof}
Since
\[
\frac{d}{dx}\left(x\exp(-c\sqrt{\log x})\right)
=
\exp(-c\sqrt{\log x})
\left(1-\frac{c}{2\sqrt{\log x}}\right),
\]
the function $x\exp(-c\sqrt{\log x})$ is
decreasing in $1\le x\le\exp((c/2)^2)$ and increasing in $\exp((c/2)^2)\le x$.
Thus, if $y\ge x>\exp((c/2)^2)$, then there is nothing to prove.
If $1\le x\le\exp((c/2)^2)$, then we have
\[
x\exp(-c\sqrt{\log x})
\le
1
=\exp\left(\frac{c^2}{4}\right)\exp\left(-\frac{c^2}{4}\right)
\le
\exp\left(\frac{c^2}{4}\right)
y\exp(-c\sqrt{\log y})
\]
since the minimum of $y\exp(-c\sqrt{\log y})$ in $y\ge1$ is taken at $y=\exp((c/2)^2)$.
\end{proof}

The following lemmas are used several times to estimate the error terms.

\begin{lemma}
\label{lem:x_p}
Let $c>0$. For $x\ge 2$, we have
\[
\sum_{p\le\sqrt{x}}\frac{x}{p}e^{-c\sqrt{\log\frac{x}{p}}}
\ll
xe^{-\frac{c}{2}\sqrt{\log x}},
\]
where the implicit constant depends only on $c$.
\end{lemma}
\begin{proof}
We have
\[
\sum_{p\le\sqrt{x}}
\frac{x}{p}e^{-c\sqrt{\log\frac{x}{p}}}
\ll
xe^{-\frac{c}{\sqrt{2}}\sqrt{\log x}}\sum_{p\le\sqrt{x}}\frac{1}{p}
\ll
xe^{-\frac{c}{2}\sqrt{\log x}}.
\]
This completes the proof.
\end{proof}

\begin{lemma}
\label{lem:r_logr}
Let $c>0$. For $1\le r\le x/4$, we have
\[
\frac{r}{\log 2r}\ll\frac{x}{\log x}e^{-c\sqrt{\log\frac{x}{r}}},
\]
where the implicit constant depends only on $c$.
\end{lemma}
\begin{proof}
For the case $1\le r\le\sqrt{x}$, the lemma follows by
\begin{equation}
\frac{r}{\log2r}
\ll
\sqrt{x}
\ll
\frac{x}{\log x}e^{-c\sqrt{\log x}}
\ll
\frac{x}{\log x}e^{-c\sqrt{\log \frac{x}{r}}}.
\end{equation}
For the case $\sqrt{x}<r\le x/4$, we have
\begin{equation}
\frac{r}{\log 2r}
\ll
\frac{r}{\log x}
=
\frac{x}{\log x}\left(\frac{x}{r}\right)^{-1}
\ll
\frac{x}{\log x}e^{-c\sqrt{\log \frac{x}{r}}}.
\end{equation}
This completes the proof.
\end{proof}

\begin{lemma}
\label{lem:x_logx_exp}
Let $c>0$. For $1+\exp(-\frac{c}{2}\sqrt{\log x})\le r\le x/4$, we have
\[
\frac{x}{\log x}e^{-c\sqrt{\log\frac{x}{r}}}
\ll
\frac{x\log r}{(\log x)^2(\log\frac{x}{r})},
\]
where the implicit constant depends only on $c$.
\end{lemma}
\begin{proof}
For $1+\exp(-\frac{c}{2}\sqrt{\log x})\le r\le\sqrt{x}$,
this lemma follows by
\begin{equation}
\frac{x}{\log x}e^{-c\sqrt{\log\frac{x}{r}}}
\ll
\frac{x}{\log x}e^{-\frac{c}{\sqrt{2}}\sqrt{\log x}}
\ll
\frac{x}{(\log x)^3}e^{-\frac{c}{2}\sqrt{\log x}}
\ll
\frac{x\log r}{(\log x)^2(\log\frac{x}{r})}.
\end{equation}
For $\sqrt{x}<r\le x/4$, we can prove the bound by
\begin{equation}
\frac{x}{\log x}e^{-c\sqrt{\log\frac{x}{r}}}
\ll
\frac{x}{(\log x)(\log\frac{x}{r})}
\ll
\frac{x\log r}{(\log x)^2(\log\frac{x}{r})}.
\end{equation}
This completes the proof.
\end{proof}

\begin{lemma}
\label{lem:loglog_diff}
For $1\le r\le x/4$, we have
\[
\log\log xr-\log\log\frac{x}{r}\ge\frac{2\log r}{\log x}
\]
and for $1\le r\le \sqrt{x}$, we have
\[
\log\log xr-\log\log\frac{x}{r}
=
\frac{2\log r}{\log x}+O\left(\left(\frac{\log r}{\log x}\right)^3\right),
\]
where the implicit constant is absolute.
\end{lemma}
\begin{proof}
This follows immediately by the Taylor expansion
\begin{align}
\log\log xr-\log\log\frac{x}{r}
&=
\log\left(1+\frac{\log r}{\log x}\right)
-
\log\left(1-\frac{\log r}{\log x}\right)\\
&=
2\sum_{\ell=0}^{\infty}
\frac{1}{2\ell+1}\left(\frac{\log r}{\log x}\right)^{2\ell+1}.
\end{align}
This completes the proof.
\end{proof}

\begin{lemma}
\label{lem:loglog_main_error_ratio}
For $1\le r\le x/4$, we have
\[
\frac{x\log r}{(\log x)^2(\log\frac{x}{r})}
\ll
\frac{x}{\log x}\left(\log\log xr-\log\log\frac{x}{r}\right)\frac{1}{\log\log x},
\]
where the implicit constant is absolute.
\end{lemma}
\begin{proof}
It is sufficient to show
\begin{equation}
\label{loglog_main_error_ratio:goal}
\frac{\log r}{\log x\log \frac{x}{r}\left(\log\log xr-\log\log\frac{x}{r}\right)}\ll\frac{1}{\log\log x}.
\end{equation}
For $1\le r\le\sqrt{x}$, we have $\log x/r\gg\log x$
so by using \cref{lem:loglog_diff}
\begin{align}
\frac{\log r}{\log x\log \frac{x}{r}(\log\log xr-\log\log \frac{x}{r})}
&\ll\frac{\log r}{(\log x)^2(\log\log xr-\log\log \frac{x}{r})}\\
&\ll\frac{1}{\log x}\ll\frac{1}{\log\log x}.
\end{align}
Therefore,
\cref{loglog_main_error_ratio:goal} holds for $1\le r\le\sqrt{x}$.
On the other hand, if $\sqrt{x}< r\le x/4$,
\begin{align}
&\frac{\log r}{\log x\log \frac{x}{r}(\log\log xr-\log\log \frac{x}{r})}\\
&\ll
\frac{1}{\log \frac{x}{r}(\log\log xr-\log\log \frac{x}{r})}\\
&\ll
\frac{1}{\log\log \frac{x}{r}(\log\log xr-\log\log \frac{x}{r})}\\
&\ll
\frac{1}{\log\log xr}\left(\frac{1}{\log\log \frac{x}{r}}+\frac{1}{\log\log xr-\log\log \frac{x}{r}}\right).
\end{align}
By $\log r\gg\log x$, $\log\log \frac{x}{r}\ge\log\log 4$ and \cref{lem:loglog_diff}, this gives
\[
\frac{\log r}{\log x\log \frac{x}{r}(\log\log xr-\log\log \frac{x}{r})}
\ll
\frac{1}{\log\log xr}
\left(\frac{1}{\log\log4}+\frac{\log x}{\log r}\right)
\ll\frac{1}{\log\log x}.
\]
This completes the proof.
\end{proof}

We recall the following forms of the prime number theorems.

\begin{lemma}[Prime number theorem]
\label{lem:PNT}
For $x\ge2$, we have
\[
\pi(x)
=
\frac{x}{\log x}
+R_0(x),\quad
\pi(x)
=
\Li(x)+R_1(x),
\]
where
\[
\pi(x)
\coloneqq
\sum_{p\le x}1,\quad
\Li(x)\coloneqq\int_{2}^{x}\frac{du}{\log u},\quad
R_0(x)\ll\frac{x}{(\log x)^2},\quad
R_1(x)\ll
xe^{-c\sqrt{\log x}},
\]
where the constant $c>0$ and the implicit constant are absolute.
\end{lemma}
\begin{proof}
See \cite[Theorem~6.9]{Montgomery_Vaughan}.
\end{proof}

\begin{lemma}[Prime number theorem]
\label{lem:PNTq}
Let $\chi\,\mod{Q}$ be a non-principal Dirichlet character.
Then, for $x\ge2$, we have
\[
\pi(x,\chi)\coloneqq\sum_{p\le x}\chi(p)
\ll
xe^{-c\sqrt{\log x}},
\]
where
the constant $c>0$ is absolute and
the implicit constant depends only on $Q$.
\end{lemma}
\begin{proof}
See \cite[Exercise~5, p.\ 383]{Montgomery_Vaughan}.
Note that if $Q>(\log x)$, then we have
\[
\pi(x,\chi)
\ll
x
\ll
e^{Q}
\ll_{Q}
1
\]
so the assertion is trivial since we allow the implicit constant to depend on $Q$.
\end{proof}

The next well-known estimate
is convenient when $r$ is close to $1$.

\begin{lemma}
\label{lem:sieve}
For any $x\ge 0$ and $h\ge2$, we have
\[
\pi(x+h)-\pi(x)
\ll
\frac{h}{\log h},
\]
where the implicit constant is absolute.
\end{lemma}
\begin{proof}
See \cite[Corollary~3.4]{Montgomery_Vaughan}.
\end{proof}

\section{Proof of Theorems \ref{thm:large_int} and \ref{thm:large} (The large $r$ case)}
In this section, we consider the case $r$ is large.
In particular, our argument in this section
has importance only in the range
\[
r_0\le r\le x/4,\quad
r_0=r_0(x)\coloneqq1+\exp(-c\sqrt{\log x}).
\]
Note that the condition $r\le x/4$ implies $\sqrt{x/r}\ge2$.
The upper bound $r\le x/4$ for $r$ is not an actual restriction since if $r>x/4$,
then by \cref{lem:to_Landau}, we see that $\pi_{2}(x;r)$ is reduced to $\pi_2(x)$
and we can apply \cref{Landau_asymp}. To prove Theorems~\ref{thm:large_int} and \ref{thm:large}, we need the following three lemmas.

\begin{lemma}
\label{lem:sum3}
For $x\ge2$, we have
\[
\sum_{p\le\sqrt{x}}\frac{p}{\log p}
=
\frac{2x}{(\log x)^2}+O\left(\frac{x}{(\log x)^3}\right)
\]
and 
\[
\sum_{p\le\sqrt{x}}\pi(p)
=
\frac{1}{2}\Li(\sqrt{x})^2
+O\left(xe^{-c\sqrt{\log x}}\right),
\]
where the constant $c>0$ and the implicit constants are absolute.
\end{lemma}
\begin{proof}
We may assume $x\ge4$.
By \cref{lem:PNT}, we have
\begin{equation}
\label{third_sum:PNT1}
\sum_{p\le\sqrt{x}}\pi(p)=\sum_{p\le\sqrt{x}}\frac{p}{\log p}+O\left(\sum_{p\le\sqrt{x}}\frac{p}{(\log p)^2}\right).
\end{equation}
This error term is estimated as
\begin{equation}
\label{third_sum:error1}
\begin{aligned}
\sum_{p\le\sqrt{x}}\frac{p}{(\log p)^2}
&\ll
\frac{\sqrt{x}}{(\log\sqrt{x})^2}\sum_{p\le\sqrt{x}}1
\ll
\frac{\sqrt{x}}{(\log\sqrt{x})^2}\frac{\sqrt{x}}{(\log\sqrt{x})}
\ll
\frac{x}{(\log x)^3}.
\end{aligned}
\end{equation}
Thus, the first assertion is an easy consequence of the latter assertion.
We have
\[
\sum_{p\le\sqrt{x}}\pi(p)
=
\sum_{n=1}^{\pi(\sqrt{x})}n
=
\frac{1}{2}(\pi(\sqrt{x})^2+\pi(\sqrt{x})).
\]
By substituting \cref{lem:PNT},
we obtain the second assertion.
\end{proof}

\begin{lemma}
\label{lem:sum2}
For $1\le r\le x/4$, we have
\[
\sum_{p\le\sqrt{x/r}}\pi(rp)
=
\int_{2}^{\sqrt{x/r}}\frac{\Li(ru)}{\log u}du
+O\left(\frac{x}{\log x}e^{-c\sqrt{\log\frac{x}{r}}}\right),
\]
where the constant $c>0$ and the implicit constant
are absolute.
\end{lemma}
\begin{proof}
On inserting \cref{lem:PNT},
\[
\sum_{p\le\sqrt{x/r}}\pi(rp)
=
\sum_{p\le\sqrt{x/r}}\Li(rp)
+
O\left(\sum_{p\le\sqrt{x/r}}rp\exp(-c\sqrt{\log rp})\right).
\]
By using \cref{lem:monotone},
this error term is estimated as
\begin{align}
\sum_{p\le\sqrt{x/r}}rp\exp(-c\sqrt{\log rp})
&\ll
\sqrt{xr}\exp(-c\sqrt{\log xr})\sum_{p\le\sqrt{x/r}}1\\
&\ll
\sqrt{xr}\exp(-c\sqrt{\log x})\sqrt{x/r}
=
x\exp(-c\sqrt{\log x}).
\end{align}
Thus,
\begin{equation}
\label{sum2:PNT}
\sum_{p\le\sqrt{x/r}}\pi(rp)
=
\sum_{p\le\sqrt{x/r}}\Li(rp)
+
O\left(x\exp(-c\sqrt{\log x})\right).
\end{equation}
By partial summation and \cref{lem:PNT},
the sum on the right-hand side above is
\begin{equation}
\label{sum2:partsum}
\begin{aligned}
\sum_{p\le\sqrt{x/r}}\Li(rp)
&=
\int_{2-}^{\sqrt{x/r}}\Li(ru)d\pi(u)\\
&=
\int_{2}^{\sqrt{x/r}}\frac{\Li(ru)}{\log u}du
+
\int_{2-}^{\sqrt{x/r}}\Li(ru)dR_1(u).
\end{aligned}
\end{equation}
For the error term,
by integrating by parts and using \cref{lem:monotone}, we have
\begin{align}
\int_{2-}^{\sqrt{x/r}}\Li(ru)dR_1(u)
&=
\Li(\sqrt{xr})R_{1}(\sqrt{x/r})
-\int_{2}^{\sqrt{x/r}}\frac{rR_{1}(u)}{\log ru}du\\
&\ll
\frac{\sqrt{xr}}{(\log\sqrt{xr})}\sqrt{x/r}e^{-c\sqrt{\log\frac{x}{r}}}
+\int_{2}^{\sqrt{x/r}}\frac{rue^{-c\sqrt{\log u}}}{\log ru}du\\
&\ll
\frac{x}{\log x}e^{-c\sqrt{\log\frac{x}{r}}}
+
\sqrt{x/r}e^{-c\sqrt{\log\frac{x}{r}}}
\int_{2}^{\sqrt{x/r}}\frac{rdu}{\log ru}.
\end{align}
The last integral is
\begin{align}
\sqrt{x/r}e^{-c\sqrt{\log\frac{x}{r}}}
\int_{2}^{\sqrt{x/r}}\frac{rdu}{\log ru}
&=
\sqrt{x/r}e^{-c\sqrt{\log\frac{x}{r}}}
\int_{2r}^{\sqrt{xr}}\frac{du}{\log u}\\
&\ll
\sqrt{x/r}e^{-c\sqrt{\log\frac{x}{r}}}\Li(\sqrt{xr})
\ll
\frac{x}{\log x}e^{-c\sqrt{\log\frac{x}{r}}}.
\end{align}
Therefore,
\[
\int_{2-}^{\sqrt{x/r}}\Li(ru)dR_1(u)
\ll
\frac{x}{\log x}e^{-c\sqrt{\log\frac{x}{r}}}.
\]
By combining this estimate with \cref{sum2:PNT} and \cref{sum2:partsum},
we arrive at the lemma.
\end{proof}

\begin{lemma}
\label{lem:sum1}
For $1\le r\le x/4$, we have
\[
\sum_{\sqrt{x/r}<p\le\sqrt{x}}\pi\left(\frac{x}{p}\right)
=
\int_{\sqrt{x/r}}^{\sqrt{x}}\frac{\Li(\frac{x}{u})}{\log u}du
+O\left(\frac{x}{\log x}e^{-c\sqrt{\log\frac{x}{r}}}\right),
\]
where the constant $c>0$ and the implicit constant are absolute.
\end{lemma}
\begin{proof}
By \cref{lem:x_p} and \cref{lem:PNT}, the above sum is
\begin{equation}
\label{sum1:PNT}
\begin{aligned}
\sum_{\sqrt{x/r}<p\le\sqrt{x}}\pi\left(\frac{x}{p}\right)
&=
\sum_{\sqrt{x/r}<p\le\sqrt{x}}\Li\left(\frac{x}{p}\right)
+O\left(\sum_{\sqrt{x/r}<p\le\sqrt{x}}
\frac{x}{p}e^{-c\sqrt{\log\frac{x}{p}}}\right)\\
&=
\sum_{\sqrt{x/r}<p\le\sqrt{x}}\Li\left(\frac{x}{p}\right)
+O\left(xe^{-c\sqrt{\log x}}\right).
\end{aligned}
\end{equation}
By partial summation, we have
\begin{equation}
\label{sum1:partsum}
\begin{aligned}
\sum_{\sqrt{x/r}<p\le\sqrt{x}}\Li\left(\frac{x}{p}\right)
&=
\int_{\sqrt{x/r}}^{\sqrt{x}}\Li\left(\frac{x}{u}\right)d\pi(u)\\
&=
\int_{\sqrt{x/r}}^{\sqrt{x}}\frac{\Li\left(\frac{x}{u}\right)}{\log u}du
+
\int_{\sqrt{x/r}}^{\sqrt{x}}\Li\left(\frac{x}{u}\right)dR_1(u).
\end{aligned}
\end{equation}
For the error term,
we use integration by parts to obtain
\begin{align}
\int_{\sqrt{x/r}}^{\sqrt{x}}\Li\left(\frac{x}{u}\right)dR_1(u)
&=
\Li(\sqrt{x})R_1(\sqrt{x})
-
\Li(\sqrt{xr})R_1(\sqrt{x/r})
+\int_{\sqrt{x/r}}^{\sqrt{x}}\frac{xR_1(u)}{u^2\log\frac{x}{u}}du\\
&\ll
\frac{x}{\log x}e^{-c\sqrt{\log\frac{x}{r}}}
+
\frac{x}{\log x}
\int_{\sqrt{x/r}}^{\sqrt{x}}\frac{e^{-c\sqrt{\log u}}}{u}du.
\end{align}
The last term is bounded as
\begin{equation}
\label{sum1:int_u_exp}
\begin{aligned}
\frac{x}{\log x}
\int_{\sqrt{x/r}}^{\sqrt{x}}\frac{e^{-c\sqrt{\log u}}}{u}du
&\ll
\frac{x}{\log x}e^{-c\sqrt{\log\frac{x}{r}}}
\int_{\sqrt{x/r}}^{\sqrt{x}}\frac{du}{u(\log u)^2}\\
&\ll
\frac{x}{\log x}e^{-c\sqrt{\log\frac{x}{r}}}
\end{aligned}
\end{equation}
by changing the value of $c$.
Therefore,
\[
\int_{\sqrt{x/r}}^{\sqrt{x}}\Li\left(\frac{x}{u}\right)dR_1(u)
\ll
\frac{x}{\log x}e^{-c\sqrt{\log\frac{x}{r}}}.
\]
On combining this estimate with \cref{sum1:partsum} and \cref{sum1:PNT},
we obtain the desired result.
\end{proof}

Now we are ready to prove Theorem~\ref{thm:large_int}.

\begin{proof}[Proof of \cref{thm:large_int}]
We just start as in the preceding works \cite{Decker_Moree,Moree_SaadEddin}.
We have
\[
\pi_{2}(x;r)=\sum_{\substack{pq\le x\\p<q\le rp}}1=\sum_{p\le x}\sum_{p<q\le\min(rp,x/p)}1.
\]
For $p>\sqrt{x}$, we have $x/p<p$ so that the inner sum is empty. Thus,
\begin{equation}
\label{large_int:doublesum}
\pi_{2}(x;r)=\sum_{p\le\sqrt{x}}\sum_{p<q\le\min(rp,x/p)}1.
\end{equation}
Since the minimum in the inner sum is determined as
\begin{equation}
\label{large_int:min}
\min\left(rp,x/p\right)
=
\left\{
\begin{array}{cl}
\displaystyle
rp&   \quad \text{if \ \ $p\le\sqrt{x/r}$},\\[1mm]
\displaystyle
x/p&   \quad \text{if\ \ $\sqrt{x/r}<p$},
\end{array}
\right.
\end{equation}
we can dissect the above expression \cref{large_int:doublesum} into three parts as
\begin{equation}
\label{large_int:decomp}
\pi_{2}(x;r)=
\sum_{\sqrt{x/r}<p\le\sqrt{x}}\pi\left(\frac{x}{p}\right)
+
\sum_{p\le\sqrt{x/r}}\pi(rp)
-
\sum_{p\le\sqrt{x}}\pi(p).
\end{equation}
Hence, by Lemmas~\ref{lem:sum3}, \ref{lem:sum2} and \ref{lem:sum1}, the function $\pi_{2}(x;r)$ is rewritten as 
\begin{equation}
\label{large_int:prefinal}
\pi_{2}(x;r)
=
G_{r}(x)+O\left(\frac{x}{\log x}e^{-c\sqrt{\log\frac{x}{r}}}\right),
\end{equation}
where
\begin{equation}
\label{large_int:G}
G_{r}(x)
\coloneqq
\int_{2}^{\sqrt{x/r}}\frac{\Li(ru)}{\log u}du
+\int_{\sqrt{x/r}}^{\sqrt{x}}\frac{\Li(\frac{x}{u})}{\log u}du
-\frac{1}{2}\Li(\sqrt{x})^2.
\end{equation}
We now apply the idea of Moree and the first author \cite{Moree_SaadEddin}
with some modification necessary for keeping the uniformity over $r$.
By taking the derivative,
\begin{align}
\frac{dG_r(x)}{dx}
&=
\frac{1}{2\sqrt{xr}}\frac{\Li(\sqrt{xr})}{\log\sqrt{x/r}}
+
\frac{1}{2\sqrt{x}}\frac{\Li(\sqrt{x})}{\log\sqrt{x}}
-
\frac{1}{2\sqrt{xr}}\frac{\Li(\sqrt{xr})}{\log\sqrt{x/r}}
-
\frac{\Li(\sqrt{x})}{2\sqrt{x}\log\sqrt{x}}\\
&\hspace{3cm}
+\int_{\sqrt{x/r}}^{\sqrt{x}}
\frac{du}{u\log u\log\frac{x}{u}}\\
&=
\frac{1}{\log x}
\int_{\sqrt{x/r}}^{\sqrt{x}}\frac{du}{u\log u}
+
\frac{1}{\log x}
\int_{\sqrt{x/r}}^{\sqrt{x}}\frac{du}{u\log\frac{x}{u}}\\
&=
\frac{1}{\log x}
\int_{\sqrt{x/r}}^{\sqrt{xr}}\frac{du}{u\log u}
=
\left(\log\log xr-\log\log\frac{x}{r}\right)\frac{1}{\log x}.
\end{align}
Therefore, by \cref{large_int:prefinal},
\begin{align}
\pi_{2}(x;r)
&=
G_r(x)-G_r(4r)+G_r(4r)+O\left(\frac{x}{\log x}e^{-c\sqrt{\log\frac{x}{r}}}\right)\\
\label{large_int:upto_G}
&=
\int_{4r}^{x}\left(\log\log ur-\log\log\frac{u}{r}\right)\frac{du}{\log u}
+G_r(4r)+O\left(\frac{x}{\log x}e^{-c\sqrt{\log\frac{x}{r}}}\right).
\end{align}
Then it suffices to consider $G_{r}(4r)$. By definition \cref{large_int:G}, we have
\begin{align}
G_r(4r)
&=
\int_{2}^{\sqrt{4r}}\frac{\Li(\frac{4r}{u})}{\log u}du
-
\frac{1}{2}\Li(\sqrt{4r})^2\\
&=
\int_{2}^{\sqrt{4r}}\frac{4r}{u\log u\log\frac{4r}{u}}du
+
O\left(
\int_{2}^{\sqrt{4r}}\frac{r}{u(\log u)(\log\frac{4r}{u})^2}du
+
\frac{r}{(\log4r)^2}\right)\\
&=
\frac{4r}{\log 4r}
\int_{2}^{\sqrt{4r}}\frac{du}{u\log u}
+
\frac{4r}{\log 4r}
\int_{2}^{\sqrt{4r}}\frac{du}{u\log\frac{4r}{u}}
+
O\left(
\frac{r\log\log4r}{(\log4r)^2}
\right)\\
&=
\frac{4r}{\log 4r}
\int_{2}^{2r}\frac{du}{u\log u}
+
O\left(
\frac{r\log\log4r}{(\log4r)^2}
\right)
=
\frac{4r\log\log4r}{\log4r}
+
O\left(\frac{r}{\log 4r}\right).
\end{align}
By substituting this into \cref{large_int:upto_G}
and using \cref{lem:r_logr}, we obtain
\[
\pi_{2}(x;r)
=
\int_{4r}^{x}\left(\log\log ur-\log\log\frac{u}{r}\right)\frac{du}{\log u}
+\frac{4r\log\log 4r}{\log 4r}
+O\left(\frac{x}{\log x}e^{-c\sqrt{\log\frac{x}{r}}}\right).
\]
This completes the proof of ~\cref{thm:large_int}.
\end{proof}

\begin{remark}
The function $G_{r}(x)$ defined in \cref{large_int:G} coincides
with the function
\begin{equation}
\label{MS:G}
G_{r}(x)
\coloneqq
\frac{1}{2}\Li(\sqrt{x})^2
-\int_{2r}^{\sqrt{xr}}\frac{\Li(\frac{t}{r})}{\log t}dt
+\int_{\sqrt{x}}^{\sqrt{xr}}\frac{\Li(\frac{x}{t})}{\log t}dt
\end{equation}
of \cite{Moree_SaadEddin}.
Indeed, we have
\begin{align}
\int_{\sqrt{x/r}}^{\sqrt{x}}
\frac{\Li\left(\frac{x}{u}\right)}{\log u}du
&=
\int_{\sqrt{x/r}}^{\sqrt{x}}\frac{du}{\log u}
\int_{2}^{x/u}\frac{dt}{\log t}
=
\int_{2}^{\sqrt{xr}}\frac{dt}{\log t}
\int_{\sqrt{x/r}}^{\min(\sqrt{x},x/t)}\frac{du}{\log u}\\
&=
\int_{\sqrt{x}}^{\sqrt{xr}}\frac{dt}{\log t}
\int_{\sqrt{x/r}}^{x/t}\frac{du}{\log u}
+
\Li(\sqrt{x})\Li(\sqrt{x})
-
\Li(\sqrt{x})\Li(\sqrt{x/r})\\
&=
\int_{\sqrt{x}}^{\sqrt{xr}}\frac{\Li(\frac{x}{t})}{\log t}dt
+
\Li(\sqrt{x})^2
-
\Li(\sqrt{xr})\Li(\sqrt{x/r})
\end{align}
and
\begin{align}
\int_{2}^{\sqrt{x/r}}\frac{\Li(ru)}{\log u}du
&=
\int_{2}^{\sqrt{x/r}}\frac{du}{\log u}
\int_{2}^{ru}\frac{dt}{\log t}
=
\int_{2}^{\sqrt{xr}}\frac{dt}{\log t}
\int_{\max(2,t/r)}^{\sqrt{x/r}}\frac{du}{\log u}\\
&=
\int_{2r}^{\sqrt{xr}}\frac{dt}{\log t}
\int_{t/r}^{\sqrt{x/r}}\frac{du}{\log u}
+
\Li(2r)\Li(\sqrt{x/r})\\
&=
-\int_{2r}^{\sqrt{xr}}\frac{\Li(\frac{t}{r})}{\log t}dt
+
\Li(\sqrt{xr})\Li(\sqrt{x/r}).
\end{align}
On inserting the above two formula into \cref{large_int:G}, we obtain \cref{MS:G}.
\end{remark}

\begin{proof}[Proof of~\cref{thm:large}]
Using integration by parts we obtain
\begin{align}
&\int_{4r}^{x}
\left(\log\log ur-\log\log\frac{u}{r}\right)\frac{du}{\log u}
+
\frac{4r\log\log4r}{\log4r}\\
&=
\Li(x)\left(\log\log xr-\log\log\frac{x}{r}\right)
-
\Li(4r)\left(\log\log4r^2-\log\log4\right)\\
&\quad+
\frac{4r\log\log4r}{\log4r}
-
\int_{4r}^{x}\left(\frac{1}{\log ur}-\frac{1}{\log\frac{u}{r}}\right)
\frac{\Li(u)}{u}du\\
&=
\frac{x}{\log x}\left(\log\log xr-\log\log\frac{x}{r}\right)
\left(1+O\left(\frac{1}{\log x}\right)\right)\\
&
\quad
+
2\int_{4r}^{x}\frac{\Li(u)\log r}{u(\log ur)(\log\frac{u}{r})}du
+O\left(\frac{r}{\log4r}\right).
\end{align}
Therefore, by using \cref{lem:r_logr} and \cref{lem:x_logx_exp},
\cref{thm:large_int} implies
\begin{equation}
\label{large:prefinal}
\begin{aligned}
\pi_2(x;r)
&=
\frac{x}{\log x}\left(\log\log xr-\log\log\frac{x}{r}\right)
\left(1+O\left(\frac{1}{\log x}\right)\right)\\
&\quad
+
2\int_{4r}^{x}\frac{\Li(u)\log r}{u(\log ur)(\log\frac{u}{r})}du
+O\left(\frac{x\log r}{(\log x)^2(\log\frac{x}{r})}\right).
\end{aligned}
\end{equation}
For the error term of the first term on the right-hand side, we have
\begin{align}
\frac{x}{(\log x)^2}\left(\log\log xr-\log\log\frac{x}{r}\right)
&=
\frac{x}{(\log x)^2}\int_{x/r}^{xr}\frac{du}{u\log u}\\
&\ll
\frac{x}{(\log x)^2(\log\frac{x}{r})}\int_{x/r}^{xr}\frac{du}{u}
\ll
\frac{x\log r}{(\log x)^2(\log\frac{x}{r})}.
\end{align}
For the integral on the right-hand side
of \cref{large:prefinal}, we have
\begin{align}
\int_{4r}^{x}\frac{\Li(u)\log r}{u(\log ur)(\log\frac{u}{r})}du
&\ll
\int_{4r}^{x}\frac{\log r}{(\log u)^2(\log\frac{u}{r})}du
\\
&\ll
(\log r)
\left(\sqrt{x}
+\frac{1}{(\log x)^2}\int_{\max(4r,\sqrt{x})}^{x}\frac{du}{\log\frac{u}{r}}\right)\\
&\ll
(\log r)
\left(\sqrt{x}
+\frac{r}{(\log x)^2}\Li\left(\frac{x}{r}\right)\right)\\
&\ll
(\log r)
\left(\sqrt{x}
+\frac{x}{(\log x)^2(\log\frac{x}{r})}\right)\\
&\ll
\frac{x\log r}{(\log x)^2(\log\frac{x}{r})}.
\end{align}
On inserting these estimates into \cref{large:prefinal}, we arrive at
\[
\pi_2(x;r)
=
\frac{x}{\log x}
\left(\log\log xr-\log\log\frac{x}{r}\right)
+
O\left(\frac{x\log r}{(\log x)^2(\log\frac{x}{r})}\right).
\]
This completes the proof of \cref{thm:large}.
\end{proof}

\section{Primes in short intervals}
In order to consider the case $r$ is very close to 1,
we need to count the number of primes in short intervals.
In this section, we recall Zaccagnini's result \cite{Zaccagnini}
on the prime number theorem in almost all short intervals
and modify his result slightly to be suitable for our application.

\begin{theorem}[{\cite[Theorem]{Zaccagnini}}]
\label{thm:Zaccagnini}
Let $\epsilon$ be a function defined over $[4,\infty)$ such that
\[
0\le\epsilon(x)\le1/6,\quad
\epsilon(x)\to 0\quad
(x\to\infty).
\]
Then, for $x^{1/6-\epsilon(x)}\le h\le x$ and $x\ge4$,
we have
\[
\int_{x}^{2x}
\left|\pi(t)-\pi(t-h)-\frac{h}{\log t}\right|^2dt
\ll
\frac{xh^2}{(\log x)^2}
\left(\epsilon(x)
+\frac{\log\log x}{\log x}\right)^2,
\]
where the implicit constant is absolute.
\end{theorem}

We shift Zaccagnini's result as follows.
\begin{lemma}
\label{lem:Z_shift}
Let $\epsilon$ be a function
defined over $[4,\infty)$ such that
\[
0\le\epsilon(x)\le1/6,\quad
\epsilon(x)\to0\quad
(x\to\infty).
\]
Then, for $x^{1/6-\epsilon(x)}\le h\le x$ and $x\ge4$, we have
\[
\int_{x}^{2x}
\left|\pi(t+h)-\pi(t)-\frac{h}{\log t}\right|^2dt
\ll
\frac{xh^2}{(\log x)^2}
\left(\epsilon(x)
+\frac{\log\log x}{\log x}\right)^2,
\]
where the implicit constant is absolute.
\end{lemma}
\begin{proof}
For small $x$, the stated bound is trivial.
Thus, we may assume that $x$ is sufficiently large.
After a change of variable in the integral, we get
\begin{equation}
\int_{x}^{2x}\left|\pi(t+h)-\pi(t)-\frac{h}{\log t}\right|^2\, dt
=
\int_{x+h}^{2x+h}\left|\pi(t)-\pi(t-h)-\frac{h}{\log (t-h)}\right|^2\, dt.
\end{equation}
Since $x^{1/6-\epsilon(x)}\le h\le x$, we have for $x+h\le t\le 2x+h$,
\begin{equation}
\label{Z_shift:log_shift}
\frac{1}{\log (t-h)}
=\frac{1}{\log t}-\frac{\log (1-h/t)}{\log t\log (t-h)}
=\frac{1}{\log t}+O\left( \frac{1}{(\log t)^2}\right).
\end{equation}
Thus,
\begin{align}
&\int_{x}^{2x}\left|\pi(t+h)-\pi(t)-\frac{h}{\log t}\right|^2dt\\
&\ll
\int_{x+h}^{2(x+h)}\left|\pi(t)-\pi(t-h)-\frac{h}{\log t}\right|^2dt
+O\left( \frac{xh^2}{(\log x)^4}\right).
\end{align}
Note that
\[
h
\ge x^{1/6-\epsilon(x)}
\ge\frac{1}{2}(x+h)^{1/6-\epsilon(x)}
=(x+h)^{1/6-\epsilon_1(x+h)},
\]
where $x_1\coloneqq x+h$ and $\epsilon_1(x_1)$ is defined
by $\epsilon_1(x_1)\coloneqq1/6$ for small $x_1$
and by $\epsilon_1(x_1)\coloneqq\epsilon(x_1-h)+\log 2/\log x_1$
for large $x_1$.
By applying \cref{thm:Zaccagnini}, we obtain
\[
\int_{x}^{2x}\left|\pi(t+h)-\pi(t)-\frac{h}{\log t}\right|^2dt
\ll
\frac{xh^2}{(\log x)^2}
\left(\epsilon(x)+\frac{\log\log x}{\log x}\right)^2.
\]
This completes the proof. 
\end{proof}

We introduce a supremum sign in \cref{lem:Z_shift}
following Saffari and Vaughan \cite{Saffari_Vaughan}.
\begin{lemma}
\label{lem:Z_sup}
Let $\epsilon$ be a function
defined over $[4,\infty)$ such that
\[
0\le\epsilon(x)\le1/6,\quad
\epsilon(x)\to0\quad
(x\to\infty).
\]
Then, for $x^{1/6-\epsilon(x)}\le H\le x$ and $x\ge4$, we have
\begin{equation}
\label{Z_sup:int}
\int_{x}^{2x}
\sup_{0\le h\le H}
\left|\pi(t+h)-\pi(t)-\frac{h}{\log t}\right|^2dt
\ll
\frac{xH^2}{(\log x)^2}
\left(\epsilon(x)
+\frac{\log\log x}{\log x}\right)^2,
\end{equation}
where the implicit constant is absolute.
\end{lemma}
\begin{proof}
For small $x$, the lemma is trivial.
Thus, we may assume that $x$ is sufficiently large.
Let $\mu\coloneqq H(\log x)^{-1}\ge x^{1/6-\epsilon(x)}(\log x)^{-1}$.
We use this parameter $\mu$ to measure the length $h$.
Since $\mu$ is independent of $h$,
this enables us to remove the dependence on $h$ in the supremum.
We take a positive integer $M$ such that
\[
(M-1)\mu<H\leq M\mu,
\]
which measures the length $H$.
For an arbitrary positive real number $h\le H$,
we take a positive integer $M(h)$ such that 
\[
\left(M(h)-1\right)\mu <h\leq M(h)\mu,
\]
which measures the length $h$.
Note that $1\leq M(h)\leq M$.
We then introduce a decomposition
\begin{align}
&\pi(t+h)-\pi(t)-\frac{h}{\log t}\\
&=
\sum_{m=1}^{M(h)}
\left(\pi\left(t+\min(m\mu,h)\right)-\pi\left( t+(m-1)\mu \right)
-\frac{\min(m\mu,h)-(m-1)\mu}{\log t}\right).
\end{align}
We next replace $\min(m\mu,h)$ by $m\mu$ in this decomposition.
The case $\min(m\mu,h)=h$ happens only when $m=M(h)$.
For $m=M(h)$, note that
$$M(h)\mu-\min(M(h)\mu,h)=M(h)\mu-h<\mu. $$
Thus, we can replace $\min(m\mu,h)=h$
by using \cref{lem:sieve} as
\begin{align}
&\pi(t+h)-\pi(t)-\frac{h}{\log t}\\
&=
\sum_{m=1}^{M(h)}
\left(\pi\left( t+m\mu\right)-\pi\left( t+(m-1)\mu \right)-\frac{\mu}{\log t}\right)
+O\left(\frac{\mu}{\log t}\right)
\end{align}
since $\log\mu\gg\log x$.
By using the Cauchy-Schwarz inequality, we get 
\begin{align}
&\left|\pi(t+h)-\pi(t)-\frac{h}{\log t}\right|^2\\
&\ll
M(h)\sum_{m=1}^{M(h)}\left|\pi\left( t+m\mu\right)-\pi\left( t+(m-1)\mu \right)-\frac{\mu}{\log t}\right|^2+\frac{\mu^2}{(\log t)^2}\\
&\ll
M\sum_{m=1}^{M}\left|\pi\left( t+m\mu\right)-\pi\left( t+(m-1)\mu \right)-\frac{\mu}{\log t}\right|^2+\frac{\mu^2}{(\log t)^2}.
\end{align}
By taking supremum over  $h$, we find that 
\begin{align}
&\sup_{0\leq h\leq H}\left|\pi(t+h)-\pi(t)-\frac{h}{\log t}\right|^2\\
&\ll
M\sum_{m=1}^{M}\left|\pi\left( t+m\mu\right)-\pi\left( t+(m-1)\mu \right)-\frac{\mu}{\log t}\right|^2+\frac{\mu^2}{(\log t)^2},
\end{align}
where we estimated the case $h=0$ trivially.  Thus, 
\begin{align}
&\int_{x}^{2x}\sup_{0\leq h\leq H}\left|\pi(t+h)-\pi(t)-\frac{h}{\log t}\right|^2dt\\
&\ll
M\sum_{m=1}^{M}\int_{x}^{2x}\left|\pi\left( t+m\mu\right)-\pi\left( t+(m-1)\mu \right)-\frac{\mu}{\log t}\right|^2dt+\frac{x\mu^2}{(\log x)^2}. 
\end{align}
By changing variable via $t+(m-1)\mu=u$ in the integral on the right-hand side,
\begin{align}
&\int_{x}^{2x}\sup_{0\leq h\leq H}\left|\pi(t+h)-\pi(t)-\frac{h}{\log t}\right|^2dt\\
&\ll
M\sum_{m=1}^{M}\int_{x+(m-1)\mu}^{2x+(m-1)\mu}\left|\pi\left( u+\mu\right)-\pi(u)-\frac{\mu}{\log (u-(m-1)\mu)}\right|^2du+\frac{x\mu^2}{(\log x)^2}\\
&\ll
M\sum_{m=1}^{M}\int_{x+(m-1)\mu}^{2x+(m-1)\mu}\left|\pi\left( u+\mu\right)-\pi(u)-\frac{\mu}{\log u}\right|^2du+\frac{M^2x\mu^2}{(\log x)^4}+\frac{x\mu^2}{(\log x)^2},
\end{align}
where we used an estimate similar to \cref{Z_shift:log_shift}.
Since $M\mu\ll H$ and $\mu=H(\log x)^{-1}$,
\begin{align}
&\int_{x}^{2x}\sup_{0\leq h\leq H}\left|\pi(t+h)-\pi(t)-\frac{h}{\log t}\right|^2dt\\
&\ll
M\sum_{m=1}^{M}\int_{x+(m-1)\mu}^{2(x+(m-1)\mu)}\left|\pi\left(u+\mu\right)-\pi( u)-\frac{\mu}{\log u}\right|^2du+\frac{xH^2}{(\log x)^4}.
\end{align}
Using the fact that $x_1\coloneqq x+(m-1)\mu\leq x+H\leq 2x$, and noting that 
\[
\mu= H(\log x)^{-1}
\ge
\frac{1}{2}(2x)^{1/6-\epsilon(x)}(\log x_1)^{-1}
\ge
\frac{1}{2}x_1^{1/6-\epsilon(x)}(\log x_1)^{-1}
=x_1^{1/6-\epsilon_1(x_1)},
\]
where the function $\epsilon_1(x_1)$
is defined by $\epsilon_1(x_1)=1/6$ for small $x_1$
and by
$\epsilon_1(x_1)\coloneqq
\epsilon(x_1-(m-1)\mu)
+
(\log\log x_1+\log 2)/\log x_1$ 
for large $x_1$.
This $\epsilon_1(x_1)$ goes to zero when $x_1\rightarrow \infty.$ Thus, by using Lemma~\ref{lem:Z_shift}, we conclude that
\[
\int_{x}^{2x}\sup_{0\leq h\leq H}\left|\pi(t+h)-\pi(t)-\frac{h}{\log t}\right|^2dt
\ll\frac{xH^2}{(\log x)^2}
\left(\epsilon(x)+\frac{\log\log x}{\log x}\right)^2.
\]
This completes the proof. 
\end{proof}

\section{The proof of Theorem~\ref{thm:small} (The small $r$ case)}
\label{section:small}
By using \cref{lem:Z_sup},
we can now deal with the case where $r$ is very close to $1$.

\begin{proof}[Proof of \cref{thm:small}]
We first consider the case
\begin{equation}
\label{small:r_away_1}
1+\exp(-c\sqrt{\log x})\le r\le\frac{3}{2},
\end{equation}
where $c$ is the same absolute constant as in \cref{thm:large}.
In this case, we just apply \cref{thm:large}.
By \cref{small:r_away_1},
we can estimate the error term of \cref{thm:large} by
\[
\frac{x\log r}{(\log x)^2(\log\frac{x}{r})}
\ll
\frac{x\log r}{(\log x)^3}
\ll
\frac{x\log r}{(\log x)^{5/2}}(\log\log x).
\]
Also, by \cref{small:r_away_1} and \cref{lem:loglog_diff},
we can rewrite the main term of \cref{thm:large} as
\begin{align}
\frac{x}{\log x}\left(\log\log xr-\log\log\frac{x}{r}\right)
&=
\frac{2x\log r}{(\log x)^2}
+
O\left(\frac{x(\log r)^3}{(\log x)^4}\right)\\
&=
\frac{2x\log r}{(\log x)^2}
+
O\left(\frac{x\log r}{(\log x)^{5/2}}(\log\log x)\right).
\end{align}
Thus, \cref{thm:large} implies the assertion provided \cref{small:r_away_1}.
Thus, we may assume
\begin{equation}
\label{small:r_cond}
1+x^{-5/12}<r\le1+\exp(-c\sqrt{\log x}).
\end{equation}
We may also assume $x$ is sufficiently large.
By \cref{large_int:doublesum} and \cref{large_int:min},
\begin{equation}
\label{small:decomp}
\pi_{2}(x;r)
=
\sum_{p\le\sqrt{x}}\sum_{p<q\le rp}1
-
\sum_{\sqrt{x/r}<p\le\sqrt{x}}\sum_{x/p<q\le rp}1
=
S_1-S_2,\quad\text{say}.
\end{equation}
For the latter sum $S_2$,
we have
\begin{equation}
\label{small:S2_ext}
S_2\ll
\sum_{\sqrt{x/r}<p\le\sqrt{x}}\sum_{p<q\le rp}1.
\end{equation}
For the inner sum,
if $(r-1)p>2x^{1/24}$, then \cref{lem:sieve} gives
\[
\sum_{p<q\le rp}1
\ll
\frac{(r-1)p}{\log((r-1)p)}
\ll
\frac{(r-1)\sqrt{x}}{\log x}
\]
and if $(r-1)p\le2x^{1/24}$, then by \cref{small:r_cond},
\[
\sum_{p<q\le rp}1
\ll
(r-1)p+1
\ll
x^{1/24}
\ll
\frac{x^{1/12}}{\log x}
\ll
\frac{(r-1)\sqrt{x}}{\log x}.
\]
Therefore, by substituting these estimates into \cref{small:S2_ext},
\begin{equation}
\label{small:S2_pre}
S_2
\ll
\frac{(r-1)\sqrt{x}}{\log x}
\sum_{\sqrt{x/r}<p\le\sqrt{x}}1.
\end{equation}
Similarly, if $(1-1/\sqrt{r})\sqrt{x}>2x^{1/24}$,
then \cref{lem:sieve} gives
\[
\sum_{\sqrt{x/r}<p\le\sqrt{x}}1
\ll
\frac{(1-1/\sqrt{r})\sqrt{x}}{\log x}
\ll
\frac{(r-1)\sqrt{x}}{\log x}
\]
and if $(1-1/\sqrt{r})\sqrt{x}\le2x^{1/24}$,
then by using the assumption \cref{small:r_cond}, we have
\[
\sum_{\sqrt{x/r}<p\le\sqrt{x}}1
\ll
\left(1-\frac{1}{\sqrt{r}}\right)\sqrt{x}+1
\ll
x^{1/24}
\ll
\frac{(r-1)\sqrt{x}}{\log x}.
\]
By substituting this estimate into \cref{small:S2_pre},
\begin{equation}
\label{small:S2}
S_2
\ll
\frac{(r-1)^2x}{(\log x)^2}.
\end{equation}
For the sum $S_1$, we decompose as
\begin{equation}
\label{small:S1_decomp}
\begin{aligned}
S_1
&=
\sum_{p\le\sqrt{x}}\left(\pi(rp)-\pi(p)\right)\\
&=
\sum_{p\le\sqrt{x}}
\frac{(r-1)p}{\log p}
+
\sum_{p\le\sqrt{x}}
\left(\pi(rp)-\pi(p)-\frac{(r-1)p}{\log p}\right)
=
S_{3}+S_{4},\quad\text{say}.
\end{aligned}
\end{equation}
We next estimate $S_{4}$.
By using the Cauchy--Schwarz inequality,
\begin{equation}
\label{small:CS}
S_{4}
\ll
\left(\frac{\sqrt{x}}{\log x}\right)^{1/2}
S_{5}^{1/2},
\end{equation}
where
\[
S_{5}
\coloneqq
\sum_{2\le n\le\sqrt{x}}
\left|\pi(rn)-\pi(n)-\frac{(r-1)n}{\log n}\right|^2.
\]
We next decompose the sum $S_{5}$ as
\begin{equation}
\label{small:S5decomp}
S_{5}
=
\sum_{2\le n\le\sqrt{x}(\log x)^{-2}}
+
\sum_{\sqrt{x}(\log x)^{-2}<n\le\sqrt{x}}
=
S_{51}+S_{52},\quad\text{say}.
\end{equation}
We estimate $S_{51}$ trivially as
\begin{equation}
\label{small:S51}
S_{51}
\ll
\frac{(r-1)^2x^{3/2}}{(\log x)^{6}}
+
\frac{\sqrt{x}}{(\log x)^2}
\ll
\frac{(r-1)^2x^{3/2}}{(\log x)^6},
\end{equation}
where we used the assumption $r\ge1+x^{-5/12}$.
For the sum $S_{52}$,
we approximate the sum by integral as follows.
For $n\le t\le n+1$, we have
\[
\pi(rn)-\pi(n)-\frac{(r-1)n}{\log n}
=
\pi(rt)-\pi(t)-\frac{(r-1)t}{\log t}
+
O(1).
\]
Therefore,
by taking the integral over $n\le t\le n+1$,
\[
\left|\pi(rn)-\pi(n)-\frac{(r-1)n}{\log n}\right|^2
\ll
\int_{n}^{n+1}
\left|\pi(rt)-\pi(t)-\frac{(r-1)t}{\log t}\right|^2
dt
+
1.
\]
By taking the summation over $n$,
we can approixmate $S_{52}$ by integral as
\begin{equation}
\label{small:S52approx}
\begin{aligned}
S_{52}
&\ll
\int_{\sqrt{x}(\log x)^{-2}}^{\sqrt{x}+1}
\left|\pi(rt)-\pi(t)-\frac{(r-1)t}{\log t}\right|^2
dt
+
\sqrt{x}\\
&\ll
\int_{\sqrt{x}(\log x)^{-2}}^{\sqrt{x}}
\left|\pi(rt)-\pi(t)-\frac{(r-1)t}{\log t}\right|^2
dt
+
(r-1)^2x
+
\sqrt{x}.
\end{aligned}
\end{equation}
We next dissect the integral dyadically.
Let $K$ be a positive integer satisfying
\[
2^{-K}\le(\log x)^{-2}<2^{-(K-1)}
\]
and let $x_{k}=\sqrt{x}/2^{k}$.
Then,
\begin{equation}
\label{small:dyadic}
\begin{aligned}
&\int_{\sqrt{x}(\log x)^{-2}}^{\sqrt{x}}
\left|\pi(rt)-\pi(t)-\frac{(r-1)t}{\log t}\right|^2
dt\\
&\ll
\sum_{k=1}^{K}
\int_{x_{k}}^{2x_k}
\left|\pi(rt)-\pi(t)-\frac{(r-1)t}{\log t}\right|^2
dt\\
&\ll
\sum_{k=1}^{K}
\int_{x_k}^{2x_k}
\sup_{0\le h\le 2(r-1)x_k}
\left|\pi(t+h)-\pi(t)-\frac{h}{\log t}\right|^2
dt.
\end{aligned}
\end{equation}
We now apply \cref{lem:Z_sup}
to each of the above integrals.
For every $k$ in the above sum,
the assumption \cref{small:r_cond} gives
\[
x_k
\ge2(r-1)x_k
\ge x^{-5/12}\sqrt{x}(\log x)^{-2}
=x^{1/12}(\log x)^{-2}
\ge x_k^{1/6-\epsilon(x_k)},
\]
where $\epsilon(x)$ is defined by
$\epsilon(x)=1/6$ for small $x$ and by
$\epsilon(x)=4\log\log x/\log x$ for large $x$.
Since this $\epsilon(x)$ satisfies the conditions
of \cref{lem:Z_sup}, we can obtain
\begin{align}
\int_{x_k}^{2x_k}
\sup_{0\le h\le 2(r-1)x_k}
\left|\pi(t+h)-\pi(t)-\frac{h}{\log t}\right|^2
dt
&\ll
\frac{(r-1)^2x_k^{3}}{(\log x_k)^2}
\left(\epsilon(x_k)
+\frac{\log\log x}{\log x}\right)^{2}\\
&\ll
\frac{1}{2^k}
\frac{(r-1)^2x^{3/2}}{(\log x)^4}(\log\log x)^2.
\end{align}
By substituting this estimate into \cref{small:dyadic}, we obtain
\[
\int_{\sqrt{x}(\log x)^{-2}}^{\sqrt{x}}
\left|\pi(rt)-\pi(t)-\frac{(r-1)t}{\log t}\right|^2
\ll
\frac{(r-1)^2x^{3/2}}{(\log x)^4}(\log\log x)^2.
\]
By combining this estimate with \cref{small:S5decomp}, \cref{small:S51} and \cref{small:S52approx},
we arrive at
\[
S_5
\ll
\frac{(r-1)^2x^{3/2}}{(\log x)^4}(\log\log x)^2.
\]
By substituting this estimate into \cref{small:CS},
we therefore find that 
\begin{equation}
\label{small:S4}
S_4
\ll
\frac{(r-1)x}{(\log x)^{5/2}}(\log\log x).
\end{equation}
We then estimate $S_3$.
By using Lemma~\ref{lem:sum3} for $S_3$, we obtain
\begin{equation}
\label{small:S3}
S_3
=
\frac{2(r-1)x}{(\log x)^2}+O\left(\frac{(r-1)x}{(\log x)^3}\right).
\end{equation}
On inserting \cref{small:S4} and \cref{small:S3} into \cref{small:S1_decomp},
\[
S_1
=
\frac{2(r-1)x}{(\log x)^2}+O\left(\frac{(r-1)x}{(\log x)^{5/2}}(\log\log x)\right)
\]
so combining this with \cref{small:decomp} and \cref{small:S2}, we deduce that
\begin{equation}
\pi_{2}(x;r)
=
\frac{2(r-1)x}{(\log x)^2}
+O\left(\frac{(r-1)^2x}{(\log x)^2}
+\frac{(r-1)x}{(\log x)^{5/2}}(\log\log x)\right).
\end{equation}
Since $\log r=(r-1)+O((r-1)^2)$ for $1\le r\le3/2$, this gives
\begin{equation}
\label{small:prefinal}
\pi_{2}(x;r)= \frac{2x\log r}{(\log x)^2}+O\left( \frac{(r-1)^2x}{(\log x)^2}+\frac{x\log r}{(\log x)^{5/2}}(\log\log x)\right).
\end{equation}
We can bound the first error term by using \cref{small:r_cond} as
\begin{equation}
\frac{(r-1)^2x}{(\log x)^2}
\ll\frac{(r-1)x}{(\log x)^2}\frac{1}{(\log x)^{1/2}}
\ll\frac{x\log r}{(\log x)^{5/2}}\log\log x.
\end{equation}
Thus, by \cref{small:prefinal},
we obtain the assertion provided \cref{small:r_cond}.
This completes the proof. 
\end{proof}

\section{Proof of \cref{thm:uniform}}
In this section,
we combine the results obtained in the preceding sections
to prove \cref{thm:uniform}, which provides an asymptotic formula
for $\pi_{2}(x;r)$ available uniformly for a wide range of $r$.

\begin{proof}[Proof of \cref{thm:uniform}]
For the case
$r_0\coloneqq1+\exp(-c\sqrt{\log x})<r\le x/4$,
we just use \cref{thm:large}.
Then, we may bound the error term of \cref{thm:large}
by \cref{lem:loglog_main_error_ratio} to obtain the theorem.
For the case $1+x^{-5/12}<r\leq r_0$,
we use \cref{thm:small} so
it is enough to show that 
\begin{equation}
\label{uniform:uniform_error_estimate_r_small1}
\frac{\log\log x}{(\log x)^{1/2} }\ll \frac{1}{\log\log x}
\end{equation}
and 
\begin{equation}
\label{uniform:uniform_error_estimate_r_small2}
\frac{2x\log r}{(\log x)^2}
=
\frac{x}{\log x}
\left(\log\log xr-\log\log\frac{x}{r}\right)
\left(1+O\left(\frac{1}{\log\log x}\right)\right). 
\end{equation}
The bound \cref{uniform:uniform_error_estimate_r_small1} is trivial.
For \cref{uniform:uniform_error_estimate_r_small2},
\cref{lem:loglog_diff} gives
\[
\log\log xr-\log\log\frac{x}{r}
=
\frac{2\log r}{\log x}
\left(1+O\left(\frac{1}{(\log x)^2}\right)\right).
\]
Then, \cref{uniform:uniform_error_estimate_r_small2}
follows immediately.
This completes the proof.
\end{proof}

\section{Bias in the distribution of the RSA integers}
\label{section:bias}
In this section,
we consider the bias in the distribution of the RSA integers
and prove \cref{thm:bias}.
We mainly follow the argument of Dummit, Granville,
and Kisilevsky \cite{Dummit_Granville_Kisilevsky}.
However,
since the resulting bias will be sensitive to the size of $r$,
we need to introduce careful treatments based on the preceding sections.
The bias will appear only for $r$ close to $x$.
Thus, we introduce a new variable $s$ by
\[
s\coloneqq\frac{x}{r}
\]
as in the statement of \cref{thm:bias}.
Also, recall that in \cref{thm:bias}, we assume \cref{PNT_assump}, \cref{delta1}, \cref{delta2} and \cref{delta3}
as given in the introduction.

\begin{proof}[Proof of \cref{thm:bias}]
In this proof, every implicit constant will depend on $Q$
even without special mention. Let
\[
\pi_{2,Q}(x;r)
\coloneqq
\sum_{\substack{pq\le x\\p<q\le rp\\(pq,Q)=1}}1,\quad
\pi_{2,\chi,\eta}(x;r)
\coloneqq
\sum_{\substack{pq\le x\\p<q\le rp\\\chi(p)=\chi(q)=\eta}}1.
\]
Similarly to \cref{large_int:doublesum},
\begin{equation}
\label{bias:first_expansion}
\pi_{2,\chi,\eta}(x;r)
=
\sum_{\substack{p\le\sqrt{x}\\\chi(p)=\eta}}
\sum_{\substack{p<q\le\min(rp,x/p)\\\chi(q)=\eta}}1.
\end{equation}
By \cref{PNT_assump}
and recalling that $\chi$ is quadratic, we have
\[
\sum_{\substack{q\le y\\\chi(q)=\eta}}1
=
\frac{1}{2}
\sum_{\substack{q\le y\\(q,Q)=1}}1
+
\frac{\eta}{2}
\sum_{q\le y}\chi(q)
=
\frac{1}{2}
\sum_{\substack{q\le y\\(q,Q)=1}}1
+
O\left(\frac{y}{\log y}\delta(y)\right)
\]
for $y\ge2$.
Thus, by writing
\[
R(x)\coloneqq\frac{x}{\log x}\delta(x)
\]
and using \cref{delta1}, we obtain
\begin{equation}
\label{bias:decomp_pre}
\begin{aligned}
\pi_{2,\chi,\eta}(x;r)
&=
\frac{1}{2}
\sum_{\substack{p\le\sqrt{x}\\\chi(p)=\eta}}
\sum_{\substack{p<q\le\min(rp,x/p)\\(q,Q)=1}}1
+
O\left(\sum_{p\le\sqrt{x}}R\left(\min\left(rp,\frac{x}{p}\right)\right)\right)\\
&=
S_1+O\left(\sum_{p\le\sqrt{x}}R\left(\min\left(rp,\frac{x}{p}\right)\right)\right),\quad\text{say}.
\end{aligned}
\end{equation}
By \cref{large_int:min},
we can decompose the error term as
\begin{equation}
\label{bias:decomp_pre_error}
\sum_{p\le\sqrt{x}}R\left(\min\left(rp,\frac{x}{p}\right)\right)
=
\sum_{p\le\sqrt{x/r}}R(rp)
+
\sum_{\sqrt{x/r}<p\le\sqrt{x}}
R\left(\frac{x}{p}\right).
\end{equation}
By using \cref{delta1}, \cref{delta2} and \cref{delta21},
the former sum is estimated as
\begin{align}
\sum_{p\le\sqrt{x/r}}R(rp)
&\ll
R(\sqrt{xr})\sum_{p\le\sqrt{x/r}}1
\ll
R(\sqrt{xr})\frac{\sqrt{x/r}}{\log s}\\
&\ll
\frac{x}{(\log x)^2(\log s)}\delta(\sqrt{x})(\log\sqrt{x})
\ll
\frac{x}{(\log x)^2}\delta(\sqrt{s}).
\end{align}
By using \cref{delta21}, the latter sum is estimated as
\begin{align}
&\sum_{\sqrt{x/r}<p\le\sqrt{x}}
R\left(\frac{x}{p}\right)
\ll
\delta(\sqrt{x})
\sum_{\sqrt{x/r}<p\le\sqrt{x}}\frac{x}{p\log\frac{x}{p}}\\
&\ll
\frac{\delta(\sqrt{x})}{\log x}
\sum_{\sqrt{x/r}<p\le\sqrt{x}}\frac{x}{p}
\ll
\left(\log\log xr-\log\log\frac{x}{r}\right)
\frac{x}{\log x}\delta(\sqrt{x}).
\end{align}
By substituting these estimates into \cref{bias:decomp_pre_error},
\begin{equation}
\label{bias:R_sum}
\begin{aligned}
&\sum_{p\le\sqrt{x}}R\left(\min\left(rp,\frac{x}{p}\right)\right)\\
&\ll
\frac{x}{(\log x)^2}\delta(\sqrt{s})
+\left(\log\log xr-\log\log\frac{x}{r}\right)\frac{x}{\log x}\delta(\sqrt{x}).
\end{aligned}
\end{equation}
Therefore, by \cref{bias:decomp_pre},
\begin{equation}
\label{bias:decomp}
\begin{aligned}
&\pi_{2,\chi,\eta}(x;r)\\
&=
S_1+O\left(\frac{x}{(\log x)^2}\delta(\sqrt{s})
+\left(\log\log xr-\log\log\frac{x}{r}\right)\frac{x}{\log x}\delta(\sqrt{x})\right).
\end{aligned}
\end{equation}
By recalling again that $\chi$ is quadratic,
\begin{equation}
\label{bias:S1}
\begin{aligned}
S_1
&=
\frac{1}{4}
\sum_{\substack{p\le\sqrt{x}\\(p,Q)=1}}
\sum_{\substack{p<q\le\min(rp,x/p)\\(q,Q)=1}}1
+
\frac{\eta}{4}
\sum_{p\le\sqrt{x}}\chi(p)
\sum_{\substack{p<q\le\min(rp,x/p)\\(q,Q)=1}}1\\
&=
\frac{1}{4}\pi_{2,Q}(x;r)+S_{2},\quad\text{say},
\end{aligned}
\end{equation}
where we used an argument similar to \cref{large_int:doublesum}.
We next decompose $S_2$ as
\begin{equation}
\label{bias:S2}
S_2
=
\frac{\eta}{4}
\sum_{p\le\sqrt{x}}\chi(p)
\sum_{\substack{q\le\min(rp,x/p)\\(q,Q)=1}}1
-
\frac{\eta}{4}
\sum_{p\le\sqrt{x}}\chi(p)
\sum_{\substack{q\le p\\(q,Q)=1}}1
=
S_3+S_4,\quad\text{say}.
\end{equation}
By \cref{PNT_assump} and integration by parts,
we obtain
\[
\pi(x)
=\frac{x}{\log x}+\Li_{2}(x)+O\left(R(x)\right),
\]
where
\[
\Li_{2}(x)\coloneqq\int_{2}^{x}\frac{du}{(\log u)^2}.
\]
Thus, the sum $S_3$ is evaluated as
\begin{equation}
\label{bias:S3_decomp}
\begin{aligned}
S_3
&=
\frac{\eta}{4}
\sum_{p\le\sqrt{x}}\chi(p)
\sum_{q\le\min(rp,x/p)}1
+O\left(\frac{\sqrt{x}}{\log x}\right)\\
&=
S_{31}+S_{32}
+
O\left(
\sum_{p\le\sqrt{x}}R\left(\min\left(rp,\frac{x}{p}\right)\right)
+
\frac{\sqrt{x}}{\log x}\right),
\end{aligned}
\end{equation}
where
\begin{equation}
S_{31}
\coloneqq
\frac{\eta}{4}
\sum_{p\le\sqrt{x}}\chi(p)
\frac{\min(rp,\frac{x}{p})}{\log\min(rp,\frac{x}{p})},\quad
S_{32}
\coloneqq
\frac{\eta}{4}
\sum_{p\le\sqrt{x}}\chi(p)
\Li_{2}\left(\min\left(rp,\frac{x}{p}\right)\right).
\end{equation}
By \cref{large_int:min}, $S_{32}$ can be rewritten as
\[
S_{32}
=
\frac{\eta}{4}
\sum_{p\le\sqrt{x/r}}\chi(p)
\Li_{2}\left(rp\right)
+
\frac{\eta}{4}
\sum_{\sqrt{x/r}<p\le\sqrt{x}}\chi(p)
\Li_{2}\left(\frac{x}{p}\right).
\]
By using \cref{PNT_assump} and \cref{delta1},
the former sum is estimated as
\begin{align}
\sum_{p\le\sqrt{x/r}}\chi(p)
\Li_{2}\left(rp\right)
&=
\int_{2-}^{\sqrt{x/r}}\Li_{2}(ru)d\pi(u,\chi)\\
&=
\Li_{2}(\sqrt{xr})\pi(\sqrt{x/r},\chi)
-
\int_{2}^{\sqrt{x/r}}\pi(u,\chi)
\frac{rdu}{(\log ru)^2}\\
&\ll
\frac{\sqrt{xr}}{(\log x)^2}R(\sqrt{x/r})
+
R(\sqrt{x/r})\int_{2r}^{\sqrt{xr}}\frac{du}{(\log u)^2}\\
&\ll
\frac{x}{(\log x)^2}\delta(\sqrt{s}).
\end{align}
The latter sum is estimated
by using \cref{PNT_assump}, \cref{delta1} and \cref{delta21} as
\begin{align}
&\sum_{\sqrt{x/r}<p\le\sqrt{x}}\chi(p)
\Li_{2}\left(\frac{x}{p}\right)
=
\int_{\sqrt{x/r}}^{\sqrt{x}}
\Li_{2}\left(\frac{x}{u}\right)d\pi(u,\chi)\\
&=
\Li_{2}(\sqrt{x})\pi(\sqrt{x},\chi)
-
\Li_{2}(\sqrt{xr})\pi(\sqrt{x/r},\chi)
+
\int_{\sqrt{x/r}}^{\sqrt{x}}
\pi(u,\chi)\frac{xdu}{u^2(\log\frac{x}{u})^2}\\
&\ll
\frac{x}{(\log x)^3}\delta(\sqrt{x})
+
\frac{x}{(\log x)^2}\delta(\sqrt{s})
+
\frac{x}{(\log x)^2}
\int_{\sqrt{s}}^{\sqrt{x}}\frac{\delta(u)}{u\log u}du
\ll
\frac{x}{(\log x)^2}\Delta(\sqrt{s}).
\end{align}
Therefore, we obtain
\[
S_{32}\ll\frac{x}{(\log x)^2}\Delta(\sqrt{s}).
\]
By substituting this into \cref{bias:S3_decomp}
and using \cref{bias:R_sum}, we obtain
\begin{equation}
\label{bias:S3}
S_3
=
S_{31}
+
O\left(\frac{x}{(\log x)^2}\Delta(\sqrt{s})
+
\left(\log\log xr-\log\log\frac{x}{r}\right)\frac{x}{\log x}\delta(\sqrt{x})\right)
\end{equation}
since \cref{delta1} and \cref{delta2} implies
\[
\frac{x}{(\log x)^2}\delta(\sqrt{s})
\gg
\frac{x}{(\log x)^2}\delta(\sqrt{x})
=
\frac{\sqrt{x}}{\log x}R(\sqrt{x})
\gg
\frac{\sqrt{x}}{\log x}.
\]
For the sum $S_4$, \cref{PNT_assump}, \cref{delta1} 
and \cref{delta21} implies
\begin{equation}
\label{bias:S4}
\begin{aligned}
S_4
&=
\frac{\eta}{4}\sum_{\substack{q\le\sqrt{x}\\(q,Q)=1}}
\sum_{q\le p\le\sqrt{x}}\chi(p)\\
&\ll
R(\sqrt{x})
\sum_{q\le\sqrt{x}}1
\ll
\frac{x}{(\log x)^2}\delta(\sqrt{x})
\ll
\frac{x}{(\log x)^2}\delta(\sqrt{s}).
\end{aligned}
\end{equation}
Combining \cref{bias:decomp}, \cref{bias:S1},
\cref{bias:S2}, \cref{bias:S3} and \cref{bias:S4},
\begin{equation}
\label{bias:abs}
\begin{aligned}
\pi_{2,\chi,\eta}(x;r)
&=
\frac{1}{4}\pi_{2,Q}(x;r)
+
S_{31}\\
&\quad+
O\left(
\frac{x}{(\log x)^2}\Delta(\sqrt{s})
+
\left(\log\log xr-\log\log\frac{x}{r}\right)\frac{x}{\log x}\delta(\sqrt{x})\right).
\end{aligned}
\end{equation}

We then divide both sides of \cref{bias:abs} by $\pi_{2,Q}(x;r)$.
To this end, we shall evaluate $\pi_{2,Q}(x;r)$.
Obviously, we have
\begin{align}
\pi_{2,Q}(x;r)
&=
\pi_{2}(x;r)
+
O\Bigg(
\sum_{\substack{p\le\sqrt{x}\\p\mid Q}}\sum_{p<q\le rp}1
+
\sum_{p\le\sqrt{x}}\sum_{\substack{p<q\le rp\\q\mid Q}}1
\Bigg)\\
&=
\pi_{2}(x;r)
+
O\left(\pi(rQ)+\frac{\sqrt{x}}{\log x}\right)
=
\pi_{2}(x;r)
+
O\left(\frac{r}{\log2r}+\frac{\sqrt{x}}{\log x}\right),
\end{align}
where the implicit constant depends on $Q$.
By the assumption $2\le r\le x/4$
and using
\cref{lem:r_logr}, \cref{lem:x_logx_exp}
and \cref{lem:loglog_main_error_ratio}
\[
\frac{r}{\log2r}
\ll
\frac{x}{\log x}
\left(\log\log xr-\log\log\frac{x}{r}\right)\frac{1}{\log\log x}.
\]
Also, by the assumption $2\le r\le x/4$ and \cref{lem:loglog_diff},
\[
\frac{\sqrt{x}}{\log x}
\ll
\frac{x}{\log x}\frac{\log r}{\log x}\frac{1}{\log\log x}
\ll
\frac{x}{\log x}
\left(\log\log xr-\log\log\frac{x}{r}\right)\frac{1}{\log\log x}.
\]
Therefore, by \cref{thm:uniform},
\[
\pi_{2,Q}(x;r)
=
\frac{x}{\log x}\left(\log\log xr-\log\log\frac{x}{r}\right)
\left(1+O\left(\frac{1}{\log\log x}\right)\right)
\]
provided $2\le r\le x/4$.
Thus, by dividing both sides of \cref{bias:abs}
by $\pi_{2,Q}(x;r)$,
\begin{equation}
\label{bias:ratio}
\frac{\pi_{2,\chi,\eta}(x;r)}{\pi_{2,Q}(x;r)}
=
\frac{1}{4}\left(1+H_{\chi,\eta}(x;r)\right),
\end{equation}
where
\begin{equation}
\begin{aligned}
H_{\chi,\eta}(x;r)
&=
\frac{\eta}{\log\log xr-\log\log\frac{x}{r}}
\left(1+O\left(\frac{1}{\log\log x}\right)\right)\\
&\quad\times
\left(\frac{\log x}{x}S_{\chi}(x;r)
+O\left(\frac{\Delta(\sqrt{s})}{\log x}\right)\right)
+
O(\delta(\sqrt{x}))
\end{aligned}
\end{equation}
with
\begin{align}
S_{\chi}(x;r)
&\coloneqq
\sum_{p\le\sqrt{x}}
\chi(p)\frac{\min(rp,\frac{x}{p})}{\log\min(rp,\frac{x}{p})}.
\end{align}
Our remaining task is to prove
\begin{equation}
\label{bias:Schi}
S_{\chi}(x;r)
=
\mathcal{L}_{\chi}(s)\frac{x}{\log x}
+O\left(\frac{x}{(\log x)^2}\Delta(\sqrt{s})\right).
\end{equation}

By \cref{large_int:min}, we can decompose $S_{\chi}(x;r)$ as
\begin{equation}
\label{bias:Schi_decomp}
S_{\chi}(x;r)
=
\frac{x}{s}
\sum_{p\le\sqrt{s}}
\chi(p)\frac{p}{\log\frac{xp}{s}}
+
\sum_{\sqrt{s}<p\le\sqrt{x}}
\chi(p)\frac{x}{p\log\frac{x}{p}}
=
S_{\chi,1}+S_{\chi,2},\quad\text{say}.
\end{equation}
For the sum $S_{\chi,1}$,
we first decompose the sum as
\[
S_{\chi,1}
\coloneqq
\frac{x}{s}
\sum_{p\le\sqrt{s}}
\chi(p)\frac{p}{\log\frac{xp}{s}}
=
\frac{x}{\log x}\frac{1}{s}
\sum_{p\le\sqrt{s}}
\chi(p)p
+
\frac{x}{\log x}\frac{1}{s}
\sum_{p\le\sqrt{s}}
\chi(p)\frac{p\log\frac{s}{p}}{\log\frac{xp}{s}}.
\]
By using \cref{PNT_assump}
and \cref{delta1},
the latter sum is estimated as
\begin{align}
\sum_{p\le\sqrt{s}}
\chi(p)\frac{p\log\frac{s}{p}}{\log\frac{xp}{s}}
&=
\int_{2-}^{\sqrt{s}}
\frac{u\log\frac{s}{u}}{\log\frac{x}{s}u}d\pi(u,\chi)\\
&=
\frac{\sqrt{s}\log\sqrt{s}}{\log\frac{x}{\sqrt{s}}}
\pi(\sqrt{s},\chi)
-
\int_{2}^{\sqrt{s}}
\frac{\pi(u,\chi)}{\log\frac{x}{s}u}
\left(\log\frac{s}{u}-1-\frac{\log\frac{s}{u}}{\log\frac{x}{s}u}\right)du\\
&\ll
\frac{s}{\log x}\delta(\sqrt{s})
+
R(\sqrt{s})\log s\int_{2}^{\sqrt{s}}\frac{du}{\log\frac{x}{s}u}\\
&\ll
\frac{s}{\log x}\delta(\sqrt{s})
+
\frac{s^{3/2}\delta(\sqrt{s})}{x}
\int_{\frac{2x}{s}}^{\frac{x}{\sqrt{s}}}
\frac{du}{\log u}
\ll
\frac{s}{\log x}\delta(\sqrt{s}).
\end{align}
Therefore, we obtain
\begin{equation}
\label{bias:Schi1}
\begin{aligned}
S_{\chi,1}
&=
\left(\frac{1}{s}\sum_{p\le\sqrt{s}}\chi(p)p\right)
\frac{x}{\log x}
+
O\left(\frac{x}{(\log x)^2}\delta(\sqrt{s})\right).
\end{aligned}
\end{equation}
For the sum $S_{\chi,2}$, we use the decomposition
\begin{equation}
\label{bias:preSchi2}
S_{\chi,2}
=
\left(\sum_{\sqrt{s}<p\le\sqrt{x}}\frac{\chi(p)}{p}\right)
\frac{x}{\log x}
+
\frac{x}{\log x}
\sum_{\sqrt{s}<p\le\sqrt{x}}\frac{\chi(p)\log p}{p\log \frac{x}{p}}.
\end{equation}
For the second sum on the right-hand side,
we use \cref{PNT_assump} and \cref{delta21} to obtain
\begin{align}
&\sum_{\sqrt{s}<p\le\sqrt{x}}\frac{\chi(p)\log p}{p\log \frac{x}{p}}
=
\int_{\sqrt{s}}^{\sqrt{x}}\frac{\log u}{u\log \frac{x}{u}}
d\pi(u,\chi)\\
&=
\frac{\pi(\sqrt{x},\chi)}{\sqrt{x}}
-
\frac{\log\sqrt{s}}{\sqrt{s}\log\frac{x}{\sqrt{s}}}
\pi(\sqrt{s},\chi)
+
\int_{\sqrt{s}}^{\sqrt{x}}
\frac{\pi(u,\chi)}{u^2\log\frac{x}{u}}
\left(\log u-1-\frac{\log u}{\log \frac{x}{u}}\right)du\\
&\ll
\frac{\delta(\sqrt{s})}{\log x}
+
\frac{1}{\log x}\int_{\sqrt{s}}^{\sqrt{x}}\frac{\delta(u)}{u}du
\ll
\frac{\Delta(\sqrt{s})}{\log x}.
\end{align}
By substituting this estimate into \cref{bias:preSchi2},
we obtain
\begin{equation}
\label{bias:Schi2_pre}
S_{\chi,2}
=
\left(\sum_{\sqrt{s}<p\le\sqrt{x}}\frac{\chi(p)}{p}\right)
\frac{x}{\log x}
+
O\left(\frac{x}{(\log x)^2}\Delta(\sqrt{s})\right).
\end{equation}
By \cref{PNT_assump} and \cref{delta21}, we have
\begin{align}
\sum_{p>\sqrt{x}}\frac{\chi(p)}{p}
&=
\int_{\sqrt{x}}^{\infty}\frac{d\pi(u,\chi)}{u}
=
-\frac{\pi(\sqrt{x},\chi)}{\sqrt{x}}
+
\int_{\sqrt{x}}^{\infty}\frac{\pi(u,\chi)}{u^2}du\\
&\ll
\frac{\delta(\sqrt{x})}{\log x}
+
\int_{\sqrt{x}}^{\infty}\frac{\delta(u)}{u\log u}du
\ll
\frac{\Delta(\sqrt{s})}{\log x}.
\end{align}
Thus, we can complete the sum on the right-hand side of
\cref{bias:Schi2_pre} to obtain
\begin{equation}
\label{bias:Schi2}
S_{\chi,2}
=
\left(\sum_{p>\sqrt{s}}\frac{\chi(p)}{p}\right)
\frac{x}{\log x}
+
O\left(\frac{x}{(\log x)^2}\Delta(\sqrt{s})\right).
\end{equation}
By combining \cref{bias:Schi_decomp}, \cref{bias:Schi1} and \cref{bias:Schi2}
and noting
\[
\frac{1}{s}\sum_{p\le\sqrt{s}}\chi(p)p
+
\sum_{p>\sqrt{s}}\frac{\chi(p)}{p}
=
\frac{1}{s}\sum_{p<\sqrt{s}}\chi(p)p
+
\sum_{p\ge\sqrt{s}}\frac{\chi(p)}{p}
=
\mathcal{L}_{\chi}(s),
\]
we get the formula \cref{bias:Schi}.
This completes the proof.
\end{proof}

\section*{Acknowledgement}
The first  author is supported by the Austrian Science Fund (FWF) : Project F5505-N26 and Project F5507-N26, which are part of the special Research Program ``Quasi Monte Carlo Methods: Theory and Application''.
A large part of this work is based on the stay of the second author
in the Institute of Financial Mathematics and Applied Number Theory,
Johannes Kepler University of Linz.
The second author would like to thank
Sumaia Saad Eddin and the staffs of the institute
for their kind support during this stay.
The second author is supported by Grant-in-Aid for JSPS Research Fellow (Grant Number: JP16J00906).



\begin{thebibliography}{99}
\bibitem{Decker_Moree}
A. Decker and P. Moree, Counting RSA integers, \textit{Results Math.} \textbf{52} (2008), 35--39.

\bibitem{Dummit_Granville_Kisilevsky}
D. Dummit. A. Granville and H. Kisilevsky, Big biases amongst products of two primes,
\textit{Mathematika} \textbf{62} (2016), 502--507.

\bibitem{Ivic}
A. Ivi\'c, \textit{The Riemann Zeta-Function}, John Wileys \& Sons, 1985.

\bibitem{Justus}
B. Justus, On integers with two prime factors,
\textit{Albanian J. Math.} \textbf{3} (2009),
189--197.


\bibitem{Landau}
E.\ Landau, Sur quelques probl\`emes relatifs \`a la distribution des numbres premiers,
\textit{Bull. Soc. Math.} France \textbf{28} (1900), 25--38.

\bibitem{Landau2}
E.\ Landau, \"Uber die Verteilung der Zahlen,
welche aus $\nu$ Primfaktoren zusammengesetzt sind,
\textit{G\"ott. Nachr. Math.-Phys. Kl.} (1911), 361--381.

\bibitem{Montgomery_Vaughan}
H. L. Montgomery and R. C. Vaughan, \textit{Multiplicative Number Theory I.Classical Theory}, Cambridge University Press, 2007.

\bibitem{Moree_SaadEddin}
P. Moree and S. Saad Eddin, Products of two proportional primes,
\textit{Int. J. Number Theory} \textbf{13} (2017), 2583-2596.

\bibitem{Saffari_Vaughan}
B. Saffari and R. C. Vaughan, On the fractional parts of $x/n$ and related sequences. II,
\textit{Ann. Inst. Fourier} (Grenoble) \textbf{27} (2) (1977), 1--30.

\bibitem{Zaccagnini}
A. Zaccagnini,
Primes in almost all short intervals,
\textit{Acta Arith.} \textbf{84} (3) (1998),
225--244.
\end{thebibliography}
\end{document}